\documentclass[reqno]{amsart}

\usepackage{amsmath}
\usepackage{amsthm}
\usepackage{amsfonts}
\usepackage{amssymb}
\usepackage{enumerate}
\usepackage{graphicx}
\usepackage{subfigure}

\newcommand{\indicator}[1]{\ensuremath{\mathbf{1}_{\{#1\}}}}
\newcommand{\oindicator}[1]{\ensuremath{\mathbf{1}_{{#1}}}}

\numberwithin{equation}{section}

\DeclareMathOperator{\var}{Var}
\DeclareMathOperator{\tr}{tr}

\DeclareMathOperator{\rank}{rank}

\newcommand{\Prob}{\mathbb{P}}
\newcommand{\E}{\mathbb{E}}
\newcommand{\C}{\mathbb{C}}

\renewcommand\Re{\operatorname{Re}}
\renewcommand\Im{\operatorname{Im}}
\newcommand{\eps}{\varepsilon}

\newcommand{\mat}{\mathbf}

\theoremstyle{plain}
  \newtheorem{theorem}{Theorem}[section]

  \newtheorem{assumption}[theorem]{Assumption}

  \newtheorem{lemma}[theorem]{Lemma}

\theoremstyle{definition}
  \newtheorem{definition}[theorem]{Definition}
  
  \newtheorem{remark}[theorem]{Remark}

\begin{document}
\title{Products of independent elliptic random matrices}

\author[S. O'Rourke]{Sean O'Rourke}
\address{Department of Mathematics, University of Colorado at Boulder, Boulder, CO 80309  }
\thanks{S. O'Rourke has been supported by grant AFOSAR-FA-9550-12-1-0083}
\email{sean.d.orourke@colorado.edu}

\author[D. Renfrew]{David Renfrew} 
\thanks{D. Renfrew is partly supported by NSF grant DMS-0838680}
\address{Department of Mathematics, UCLA  }
\email{dtrenfrew@math.ucla.edu}

\author[A. Soshnikov]{Alexander Soshnikov}
\address{Department of Mathematics, University of California, Davis, One Shields Avenue, Davis, CA 95616-8633  }
\thanks{A. Soshnikov has been supported in part by NSF grant DMS-1007558}
\email{soshniko@math.ucdavis.edu}

\author{Van Vu}
\thanks{V. Vu is supported by research grants DMS-0901216, DMS-1307797, and AFOSAR-FA-9550-12-1-0083.}
\address{Department of Mathematics, Yale University, New Haven, CT 06520, USA}
\email{van.vu@yale.edu}

\begin{abstract}
For fixed $m > 1$, we study the product of $m$ independent $N \times N$ elliptic random matrices as $N$ tends to infinity. Our main result shows  that the empirical spectral distribution of the product converges, with probability $1$, to the $m$-th power of the circular law, regardless of the joint distribution of the mirror entries in each matrix. This leads to a new kind of  universality phenomenon: the limit law for the product of independent random matrices is independent of the limit laws for the individual matrices themselves.  

Our result also generalizes earlier results of  G\"otze--Tikhomirov \cite{GTprod} and O'Rourke--Soshnikov \cite{OS} concerning the product of independent iid random matrices.  \end{abstract}

\maketitle

\section{Introduction}


We begin by recalling that the \emph{eigenvalues} of a $N \times N$ matrix $\mat{M}$ are the roots in $\mathbb{C}$ of the characteristic polynomial $\det( \mat{M} - z \mat{I})$, where $\mat{I}$ is the identity matrix.  We let $\lambda_1(\mat{M}), \ldots, \lambda_N(\mat{M})$ denote the eigenvalues of $\mat{M}$.  In this case, the \emph{empirical spectral measure} $\mu_{\mat{M}}$ is given by
$$ \mu_\mat{M} := \frac{1}{N} \sum_{i=1}^N \delta_{\lambda_i(\mat{M})}. $$
The corresponding \emph{empirical spectral distribution} (ESD) is given by
$$ F^{\mat{M}}(x,y) := \frac{1}{N} \# \left\{1 \leq i \leq N: \Re(\lambda_i(\mat{M})) \leq x, \Im(\lambda_i(\mat{M})) \leq y \right\}. $$
Here $\# E$ denotes the cardinality of the set $E$.  

If the matrix $\mat{M}$ is Hermitian, then the eigenvalues $\lambda_1(\mat{M}), \ldots, \lambda_N(\mat{M})$ are real.  In this case the ESD is given by 
$$ F^{\mat{M}}(x) := \frac{1}{N} \# \left\{ 1 \leq i \leq N : \lambda_i(\mat{M}) \leq x \right\}. $$


One of the simplest random matrix ensembles is the class of random matrices with independent and identically distributed (iid) entries.  

\begin{definition}[iid random matrix]
Let $\xi$ be a complex random variable.  We say $\mat{Y}_N$ is an $N \times N$ \emph{iid random matrix} with atom variable $\xi$ if the entries of $\mat{Y}_N$ are iid copies of $\xi$.  
\end{definition}

When $\xi$ is a standard complex Gaussian random variable, $\mat{Y}_N$ can be viewed as a random matrix drawn from the probability distribution
$$ \Prob(d \mat{M}) = \frac{1}{\pi^{N^2}} e^{- \tr (\mat{M} \mat{M}^\ast)} d \mat{M} $$
on the set of complex $N \times N$ matrices.  Here $d \mat{M}$ denotes the Lebesgue measure on the $2N^2$ real entries
$$ \{ \Re(m_{ij}) : 1 \leq i, j \leq N\} \cup \{ \Im(m_{ij}) : 1 \leq i, j \leq N\} $$
of $\mat{M}=(m_{ij})_{i,j=1}^N$.  
The measure $\Prob(d \mat{M})$ is known as the {\it complex Ginibre ensemble}.  The {\it real Ginibre ensemble} is defined analogously.  Following Ginibre \cite{Gi}, one may compute the joint density of the eigenvalues of a random matrix $\mat{Y}_N$ drawn from the complex Ginibre ensemble.  Indeed, $(\lambda_1(\mat{Y}_N), \ldots, \lambda_N(\mat{Y}_N))$ has density
\begin{equation} \label{eq:ginibre}
	p_N(z_1, \ldots, z_N) := \frac{1}{\pi^N \prod_{i=1}^N k! } \exp \left( - \sum_{k=1}^N |z_k|^2 \right) \prod_{1 \leq i < j \leq N} |z_i - z_j |^2. 
\end{equation}

Mehta  \cite{M,M:B} used the joint density function \eqref{eq:ginibre} to compute the limiting spectral measure of the complex Ginibre ensemble.  In particular, he showed that if $\mat{Y}_N$ is drawn from the complex Ginibre ensemble, then the ESD of $\frac{1}{\sqrt{N}} \mat{Y}_N$ converges to the {\it circular law} $F_{\mathrm{circ}}$ as $N \to \infty$, where
$$ F_{\mathrm{circ}}(x,y) := \mu_{\mathrm{circ}} \left( \left\{ z \in \mathbb{C} : \Re(z) \leq x, \Im(z) \leq y \right\} \right) $$
and $\mu_{\mathrm{circ}}$ is the uniform probability measure on the unit disk in the complex plane.  Edelman \cite{Ed-cir} verified the same limiting distribution for the real Ginibre ensemble.

For the general (non-Gaussian) case, there is no formula for the joint distribution of the eigenvalues and the problem appears much more difficult.  The universality phenomenon in random matrix theory asserts that the spectral behavior of an iid random matrix does not depend on the distribution of the atom variable $\xi$ in the limit $N \rightarrow \infty$.  In other words, one expects that the circular law describes the limiting ESD of a large class of random matrices (not just Gaussian matrices).  

An important result was obtained by Girko \cite{G1,G2} who related the empirical spectral measure of a non-Hermitian matrix to that of a family of Hermitian matrices.  Using this Hermitization technique, Bai \cite{Bcirc,BSbook} gave the first rigorous proof of the circular law for general (non-Gaussian) distributions.  He proved the result under a number of moment and smoothness assumptions on the atom variable $\xi$, and a series of recent improvements were obtained by  G\"otze and Tikhomirov \cite{GTcirc},  
Pan and Zhou \cite{PZ} and  Tao and Vu \cite{TVcirc, TVesd}.   In particular,  Tao and Vu \cite{TVbull,TVesd} established the law with the minimum assumption  that $\xi$ has finite variance.

\begin{theorem}[Tao-Vu, \cite{TVesd}] \label{thm:tvcirc}
Let $\xi$ be a complex random variable with mean zero and unit variance.  For each $N \geq 1$, let $\mat{Y}_N$ be a $N \times N$ iid random matrix with atom variable $\xi$. Then the ESD of $\frac{1}{\sqrt{N}} \mat{Y}_N$ converges almost surely to the circular law $F_{\mathrm{circ}}$ as $N \rightarrow \infty$.  
\end{theorem}

More recently, G\"{o}tze and Tikhomirov \cite{GTprod} consider the ESD of the product of $m$ independent iid random matrices.  They show that, as the sizes of the matrices tend to infinity, the limiting distribution is given by $F_m$, where $F_m$ is supported on the unit circle in the complex plane and has density $f_m$ given by
\begin{equation}
\label{plotnost}
	f_m(z) := \left\{
     			\begin{array}{ll}
       				\frac{1}{m \pi} |z|^{\frac{2}{m}-2}, & \text{ for } |z|\leq 1 \\
       				0, & \text{ for } |z| > 1
     			\end{array} \right.
\end{equation}
in the complex plane.  It can be verified directly, that if $\psi$ is a random variable distributed uniformly on the unit disk in the complex plane, then $\psi^m$ has distribution $F_m$.  

\begin{theorem}[G\"{o}tze-Tikhomirov, \cite{GTprod}] \label{thm:GTprod}
Let $m \geq 1$ be an interger, and assume $\xi_1, \ldots, \xi_m$ are complex random variables with mean zero and unit variance.  For each $N \geq 1$ and $1 \leq k \leq m$, let $\mat{Y}_{N,k}$ be an $N \times N$ iid random matrix with atom variable $\xi_k$, and assume $\mat{Y}_{N,1}, \ldots, \mat{Y}_{N,m}$ are independent.  Define the product 
$$ \mat{P}_N := N^{-m/2} \mat{Y}_{N,1} \cdots \mat{Y}_{N,m}. $$
Then $\E F^{\mat{P}_N}$ converges to $F_m$ as $N \to \infty$.  
\end{theorem}

The convergence of $F^{\mat{P}_N}$ to $F_m$ in Theorem \ref{thm:GTprod} was strengthened to almost sure convergence in \cite{B,OS}.  The Gaussian case was originally considered by Burda, Janik, and Waclaw \cite{BJW}; see also \cite{Bsurv}.  We refer the reader to \cite{AB,ABK, AIK, AIK2, AKW, AS, BJLNS, F, F2} and references therein for many other interesting results concerning products of Gaussian random matrices.

\section{New results}

In this paper, we generalize Theorem \ref{thm:GTprod} by considering products of independent real elliptic random matrices.  Elliptic random matrices were originally introduced by Girko \cite{Gorig, Gten} in the 1980s.

\begin{definition}[Real elliptic random matrix] \label{def:elliptic}
Let $(\xi_1, \xi_2)$ be a random vector in $\mathbb{R}^2$, and let $\zeta$ be a real random variable.  We say $\mat{Y}_N = (y_{ij})_{i,j=1}^N$ is a $N \times N$ \emph{real elliptic random matrix} with atom variables $(\xi_1,\xi_2), \zeta$ if the following conditions hold.  
\begin{itemize}
\item (independence) $\{ y_{ii} : 1 \leq i \leq N\} \cup \{ (y_{ij}, y_{ji}) : 1 \leq i < j \leq N\}$ is a collection of independent random elements.
\item (off-diagonal entries) $\{ (y_{ij}, y_{ji}) : 1 \leq i < j \leq N\}$ is a collection of iid copies of $(\xi_1,\xi_2)$.
\item (diagonal entries) $\{y_{ii} : 1 \leq i \leq N\}$ is a collection of iid copies of $\zeta$.  
\end{itemize}
\end{definition}

Real elliptic random matrices generalize iid random matrices.  Indeed, if $\xi_1, \xi_2, \zeta$ are iid, then $\mat{Y}_N$ is just an iid random matrix.  On the other hand, if $\xi_1 = \xi_2$ almost surely, then $\mat{Y}_N$ is a real symmetric matrix.  In this case, the eigenvalues of $\mat{Y}_N$ are real and $\mat{Y}_N$ is known as a \emph{real symmetric Wigner matrix} \cite{W}.  

Suppose $\xi_1, \xi_2$ have mean zero and unit variance.  Set $\rho := \E[\xi_1 \xi_2]$.  When $|\rho| < 1$ and $\zeta$ has mean zero and finite variance, it was shown in \cite{NgO} that the ESD of $\frac{1}{\sqrt{N}} \mat{Y}_N$ converges almost surely to the elliptic law $F_\rho$ as $N \to \infty$, where 
$$ F_\rho(x,y) = \mu_{\rho}( \{z \in \mathbb{C} : \Re(z) \leq x, \Im(z) \leq y\}) $$
and $\mu_\rho$ is the uniform probability measure on the ellipsoid 
$$ \mathcal{E}_{\rho} = \left\{ z \in \C : \frac{\Re(z)^2}{(1+\rho)^2} +\frac{\Im(z)^2}{(1-\rho)^2} < 1 \right\} .$$
This is a natural generalization of the circular law (Theorem \ref{thm:tvcirc}).  Figure \ref{fig:elliptic} displays a numerical simulation of the eigenvalues of a real elliptic random matrix.  

\begin{figure} 
	\begin{center}
	\subfigure{
	\includegraphics[width=9cm]{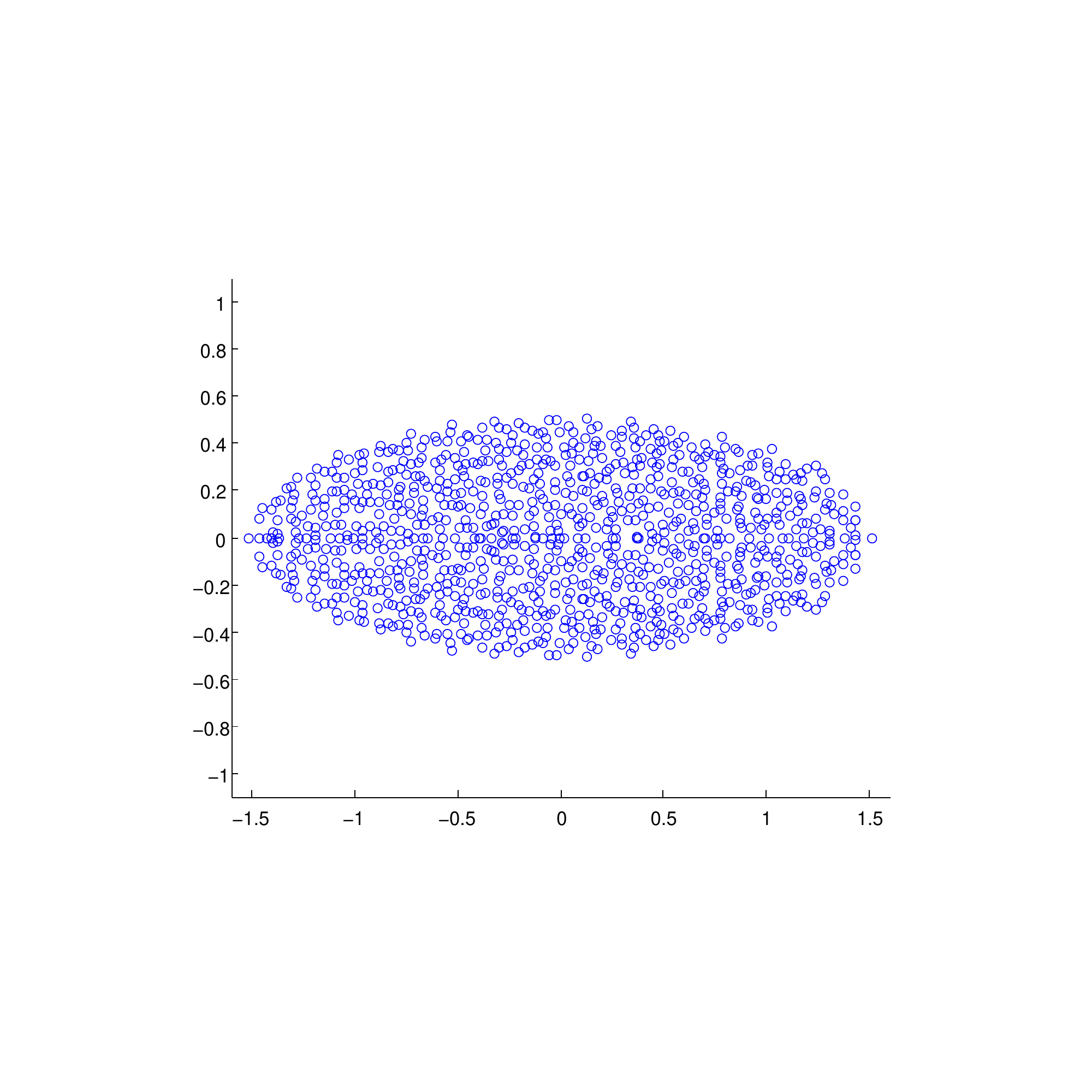}}
   	\caption{The eigenvalues of a $1000 \times 1000$ elliptic random matrix with Gaussian entries when $\rho=1/2$.} 
   	\label{fig:elliptic}
 	\end{center}
\end{figure}

In this note, we consider the product $\mat{Y}_{N,1} \cdots \mat{Y}_{N,m}$ of $m$ independent real elliptic random matrices.  In particular, we assume each real elliptic random matrix $\mat{Y}_{N,k}$ has atom variables $(\xi_{k,1}, \xi_{k,2}), \zeta_k$ which satisfy the following conditions.  


\begin{assumption} \label{ass:atom}
There exists $\tau > 0$ such that the following conditions hold.
\begin{enumerate}
\item $\xi_{k,1}, \xi_{k,2}$ both have mean zero and unit variance.  
\item $\E| \xi_{k,1}|^{2 + \tau} + \E|\xi_{k,2}|^{2 + \tau} < \infty$. \label{item:tau}
\item $\rho_k := \E[\xi_{k,1} \xi_{k,2}]$ satisfies $|\rho_k| < 1$. \label{item:rhoass}
\item $\zeta_k$ has mean zero and finite variance.
\end{enumerate}
\end{assumption}

In our main result below, we show that the limiting distribution $F_m$ (with density given by \eqref{plotnost}) from Theorem \ref{thm:GTprod} for the product of independent iid random matrices is also the limiting distribution for the product of independent elliptic random matrices.  In other words, the limit law for the product of independent random matrices is independent of the limit laws for the individual matrices themselves.  This type of universality was first considered by Burda, Janik, and Waclaw in \cite{BJW} for matrices with Gaussian entries; see also \cite{Bsurv}.  Figure \ref{fig:plots} displays several numerical simulations which illustrate this phenomenon.  

\begin{theorem} \label{thm:simple}
Let $m > 1$ be an integer.  For each $1 \leq k \leq m$, let $(\xi_{k,1}, \xi_{k,2}), \zeta_k$ be real random elements that satisfy Assumption \ref{ass:atom}.  For each $N \geq 1$ and $1 \leq k \leq m$, let $\mat{Y}_{N,k}$  be an $N \times N$ real elliptic random matrix with atom variables $(\xi_{k,1}, \xi_{k,2}), \zeta_k$, and assume $\mat{Y}_{N,1}, \ldots, \mat{Y}_{N,m}$ are independent.  Then the ESD of the product
$$ \mat{P}_N := N^{-m/2} \mat{Y}_{N,1} \cdots \mat{Y}_{N,m} $$
converges almost surely to $F_m$ (with density given by \eqref{plotnost}) as $N \to \infty$.
\end{theorem}

More generally, we establish a version of Theorem \ref{thm:simple} where each elliptic random matrix $\mat{Y}_{N,k}$ is perturbed by a deterministic, low rank matrix $\mat{A}_{N,k}$ with small Hilbert-Schmidt norm.  In fact, Theorem \ref{thm:simple} will follow from Theorem \ref{thm:main} below.  We recall that, for any $m \times n$ matrix $\mat{M}$, the Hilbert-Schmidt norm $\|\mat{M}\|_2$ is given by the formula
\begin{equation} \label{eq:def:hs}
	\|\mat{M}\|_2 := \sqrt{ \tr (\mat{M} \mat{M}^\ast) } = \sqrt{ \tr (\mat{M}^\ast \mat{M})}. 
\end{equation}

\begin{theorem} \label{thm:main}
Let $m > 1$ be an integer.  For each $1 \leq k \leq m$, let $(\xi_{k,1}, \xi_{k,2}), \zeta_k$ be real random elements that satisfy Assumption \ref{ass:atom}.  For each $N \geq 1$ and $1 \leq k \leq m$, let $\mat{Y}_{N,k}$  be an $N \times N$ real elliptic random matrix with atom variables $(\xi_{k,1}, \xi_{k,2}), \zeta_k$, and assume $\mat{Y}_{N,1}, \ldots, \mat{Y}_{N,m}$ are independent.  For each $1 \leq k \leq m$, let $\mat{A}_{N,k}$ be a $N \times N$ deterministic matrix, and assume
\begin{equation} \label{eq:Aassump}
	\max_{1 \leq k \leq m} \rank(\mat{A}_{N,k}) = O(N^{1-\eps}) \quad \text{and} \quad \sup_{N \geq 1} \max_{1 \leq k \leq m} \frac{1}{N^2} \| \mat{A}_{N,k} \|_2 < \infty, 
\end{equation}
for some $\eps > 0$.  Then the ESD of the product
\begin{equation} \label{def:PN}
	\mat{P}_N := N^{-m/2} \prod_{k=1}^m ( \mat{Y}_{N,k} + \mat{A}_{N,k}) 
\end{equation}
converges almost surely to $F_m$ (with density given by \eqref{plotnost}) as $N \to \infty$.
\end{theorem}

\begin{remark}
We conjecture that items \eqref{item:tau} and \eqref{item:rhoass} from Assumption \ref{ass:atom} are not required for Theorem \ref{thm:main} to hold.  Indeed, in view of Theorem \ref{thm:tvcirc} and \cite{NgO}, it is natural to conjecture that $\xi_1, \xi_2$ need only have two finite moments.  Also, our proof of Theorem \ref{thm:main} can almost be completed under the assumption that $-1 < \rho_k \leq 1$.  We only require that $\rho_k \neq 1$ in Section \ref{sec:singularvalue} in order to control the least singular value of matrices of the form $\mat{Y}_{N,k} + \mat{F}_{N,k}$, where $\mat{F}_{N,k}$ is a deterministic matrix whose entries are bounded by $N^{\alpha}$, for some $\alpha > 0$.  See Remark \ref{rem:lsv} and Theorem \ref{thm:prodWigner} below for further details.  
\end{remark}

\begin{remark}
Among other things, the perturbation by $\mat{A}_{N,k}$ in Theorem \ref{thm:main} allows one to consider elliptic random matrices with nonzero mean.  Indeed, let $\mu_k$ be a real number, and assume each entry of $\mat{A}_{N,k}$ takes the value $\mu_k$.  Then $\mat{Y}_{N,k} + \mat{A}_{N,k}$ is an elliptic random matrix whose atom variables have mean $\mu_k$.  
\end{remark}

\begin{remark}
In \cite{GNT}, a result similar to Theorem \ref{thm:simple} is proved under a different set of assumptions.
\end{remark}

As noted above, when $\rho_k = 1$, the matrix $\mat{Y}_{N,k}$ is known as a real symmetric Wigner matrix.  Theorem \ref{thm:main} requires that $|\rho_k| < 1$, but in the special case when $m=2$, we are able to extend our proof to show that the same result holds for the product of two independent real symmetric Wigner matrices.  

\begin{theorem} \label{thm:prodWigner}
Let $\xi_{1,1}, \xi_{2,1}$ be real random variables with mean zero and unit variance, and which satisfy 
$$ \E|\xi_{1,1}|^{2+\tau} + \E|\xi_{2,1}|^{2+\tau} < \infty $$
for some $\tau > 0$.  For each $N \geq 1$ and $k=1,2$, let $\mat{Y}_{N,k}$ be an $N \times N$ real symmetric matrix whose diagonal entries and upper diagonal entries are iid copies of $\xi_{k,1}$, and assume $\mat{Y}_{N,1}$ and $\mat{Y}_{N,2}$ are independent.  Then the ESD of the product
$$ \mat{P}_N := N^{-1} \mat{Y}_{N,1} \mat{Y}_{N,2} $$
converges almost surely to $F_2$ (with density given by \eqref{plotnost} when $m=2$) as $N \to \infty$.
\end{theorem}

\begin{figure} 
	\begin{center}
	\subfigure{
	\includegraphics[width=5cm]{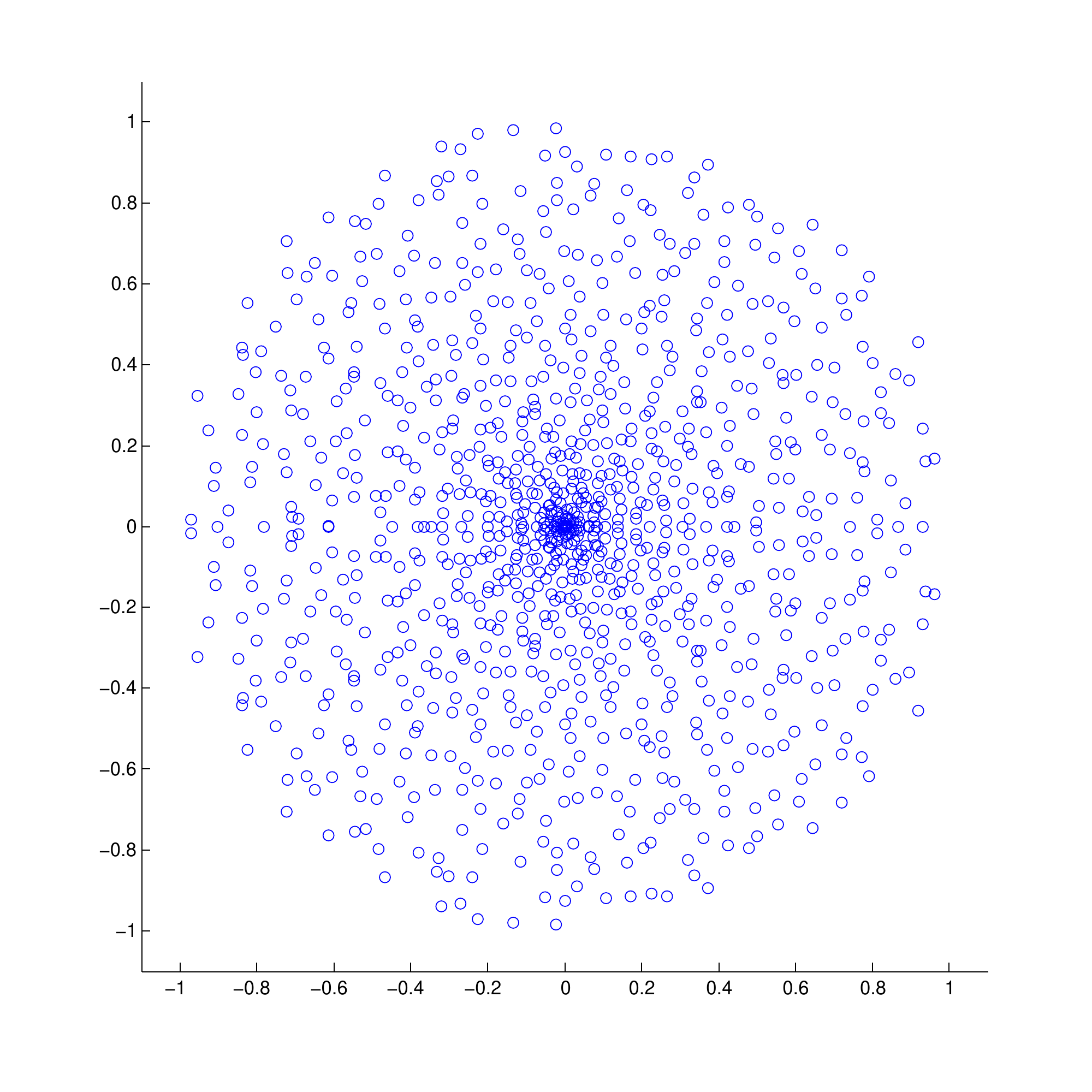}}
	\quad
	\subfigure{
  \includegraphics[width=5cm]{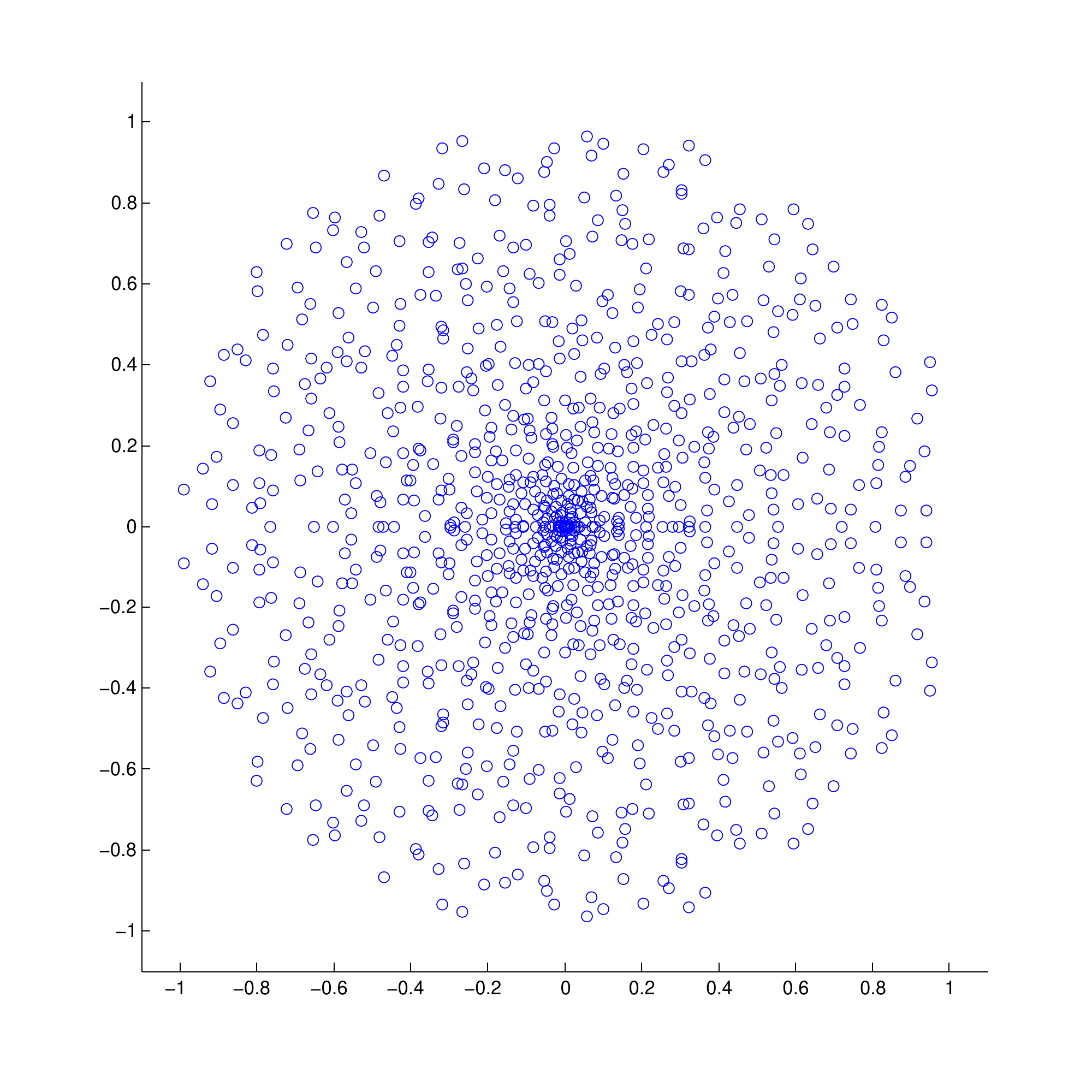}}
  \quad
  \subfigure{
    \includegraphics[width=5cm]{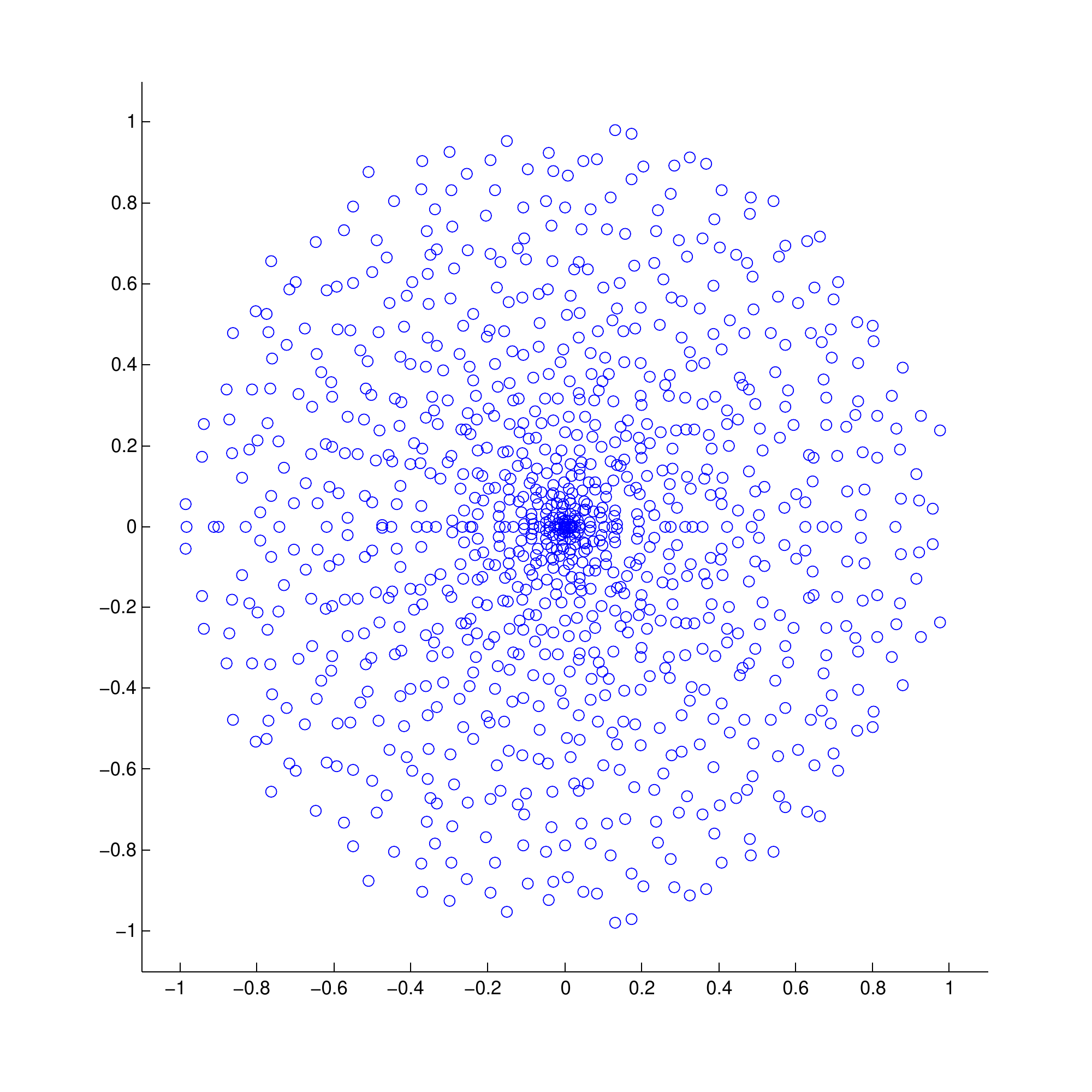}}
\quad
  \subfigure{
    \includegraphics[width=5cm]{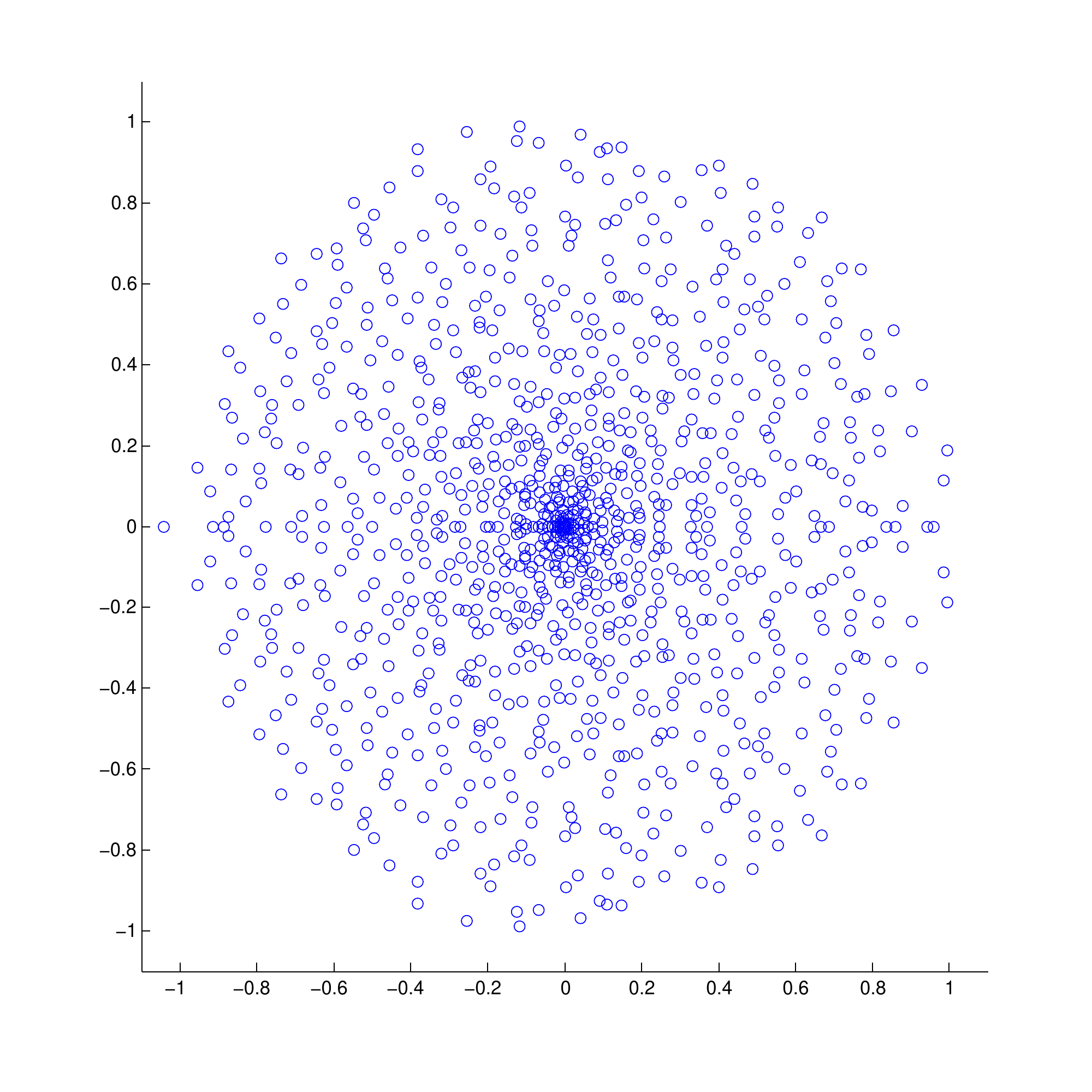}}
   	\caption{The plots show the eigenvalues of the product of two independent $1000 \times 1000$ elliptic random matrices. The plot in the upper-left corner shows the eigenvalues of the product of two identically distributed elliptic random matrices with Gaussian entries when $\rho_1 = \rho_2=1/2$.  The upper-right plot depicts the eigenvalues of the product of two independent Wigner matrices.  The bottom-left plot shows the eigenvalues of the product of two iid random matrices. The plot in the bottom-right corner contains the eigenvalues of the product of a Wigner matrix and an independent iid random matrix.}
   	\label{fig:plots}
 	\end{center}
\end{figure}

\subsection{Overview and outline}
We begin by outlining the proof of Theorem \ref{thm:main}.  Instead of directly considering $ \mat{P}_N := N^{-m/2} \prod_{k=1}^m ( \mat{Y}_{N,k} + \mat{A}_{N,k})$, we introduce a linearized random matrix, $\mat{Z}_N:= \frac{1}{\sqrt{N}} (\mat{Y}_N + \mat{A}_N)$, where $\mat{Y}_N$ and $\mat{A}_N$ are $mN \times mN$ block matrices of the form
\begin{equation} \label{def:YN}
	\mat{Y}_N := \begin{bmatrix} 
				0 &  \mat{Y}_{N,1} &     &            & 0       \\
                         		0 & 0    & \mat{Y}_{N,2} &            & 0        \\      
                           		&      & \ddots & \ddots     &         \\
                         		0 &      &     &          0 & \mat{Y}_{N,m-1} \\
                       		\mat{Y}_{N,m} &      &     &            &  0      
	\end{bmatrix} 
\end{equation}
and
\begin{equation} \label{def:AN}
	\mat{A}_N := \begin{bmatrix} 
				0 &  \mat{A}_{N,1} &     &            & 0       \\
                         		0 & 0    & \mat{A}_{N,2} &            & 0        \\      
                           		&      & \ddots & \ddots     &         \\
                         		0 &      &     &          0 & \mat{A}_{N,m-1} \\
                       		\mat{A}_{N,m} &      &     &            &  0      
			\end{bmatrix}. 
\end{equation}

The following theorem gives the limiting distribution of $\mat Z_N$, from which we will deduce our main theorem as a corollary. 

\begin{theorem} \label{thm:circular}
Under the assumptions of Theorem \ref{thm:main}, the ESD of $\mat{Z}_N := \frac{1}{\sqrt{N}} (\mat{Y}_N + \mat{A}_N)$ converges almost surely to the circular law $F_{\mathrm{circ}}$ as $N \to \infty$.
\end{theorem}

In Section \ref{sec:linear}, we show that Theorem \ref{thm:main} is a short corollary of Theorem \ref{thm:circular}.  This same linearization trick was used in \cite{OS} to study products of non-Hermitian matrices with iid entries.  Similar techniques were also used in \cite{A,HT} to study general self-adjoint polynomials of self-adjoint random matrices. 

Sections \ref{sec:resolvent}, \ref{sec:singularvalue}, and \ref{sec:complete} are dedicated to proving Theorem \ref{thm:circular}.  Following the ideas of Girko \cite{G1,G2}, we compute the limiting spectral measure of a non-Hermitian random matrix $\mat{M}$, by employing the method of \emph{Hermitizaition}.  Given an $N \times N$ matrix $\mat{M}$, we recall that the empirical spectral measure of $\mat{M}$ is given by
$$ \mu_\mat{M} := \frac{1}{N} \sum_{i=1}^N \delta_{\lambda_i(\mat{M})}, $$ 
where $\lambda_1(\mat{M}), \ldots, \lambda_N(\mat{M})$ are the eigenvalues of $\mat{M}$.  We let $\nu_{\mat{M}}$ denote the symmetric empirical measure built from the singular values of $\mat{M}$.  That is,
$$ \nu_{\mat{M}} := \frac{1}{2N} \sum_{i =1}^N \left( \delta_{\sigma_i(\mat{M})} + \delta_{-\sigma_i(\mat{M})} \right), $$ 
where $\sigma_1(\mat{M}) \geq \cdots \geq \sigma_N(\mat{M}) \geq 0$ are the singular values of $\mat{M}$.  In particular, 
$$ \sigma_1(\mat{M}) := \sup_{\|x\| = 1} \| \mat{M} x \| $$ 
is the largest singular value of $\mat{M}$ and 
$$ \sigma_N(\mat{M}) := \inf_{\|x\| = 1} \| \mat{M} x \| $$ 
is the smallest singular value, both of which will play a key role in our analysis below.  

The key observation of Girko \cite{G1,G2} relates the empirical spectral measure of a non-Hermitian matrix to that of a Hermitian matrix.  To illustrate the connection, consider the \emph{Cauchy--Stieltjes transform} $s_N$ of the measure $\mu_{\mat{M}}$, where $\mat{M}$ is an $N \times N$ matrix, given by
$$ s_N(z) := \frac{1}{N} \sum_{i=1}^N \frac{1}{ \lambda_i(\mat{M}) - z} = \int_{\mathbb{C}} \frac{1}{x -z } \mu_{\mat{M}} (dx), $$
for $z \in \mathbb{C}$.  Since $s_N$ is analytic everywhere except at the poles (which are exactly the eigenvalues of $\mat{M}$), the real part of $s_N$ determines the eigenvalues.  Let $\sqrt{-1}$ denote the imaginary unit, and set $z = s + \sqrt{-1} t$.  Then we can write the real part of $s_N(z)$ as 
\begin{align*}
	\Re( s_N(z)) &= \frac{1}{N} \sum_{i=1}^N \frac{ \Re(\lambda_i(\mat{M})) - s }{ \left| \lambda_i(\mat{M}) - z \right|^2} \\
		&= \frac{-1}{2N} \sum_{i=1}^N \frac{\partial}{\partial s} \log \left| \lambda_i(\mat{M}) - z \right|^2 \\
		&= \frac{-1}{2N} \frac{\partial}{\partial s}  \log \det \left( \mat{M} - z \mat{I} \right) \left( \mat{M} - z\mat{I} \right)^{\ast} \\
		&= -\frac{\partial}{\partial s} \int_{\mathbb{C}} \log x^2  \nu_{\mat{M} - z \mat{I}}(dx), 
\end{align*}
where $\mat{I}$ denotes the identity matrix.  In other words, the task of studying $\mu_{\mat{M}}$ reduces to studying the measures $\{\nu_{\mat{M} - z \mat{I}}\}_{z \in \mathbb{C}}$.  The difficulty now is that the $\log$ function has two poles, one at infinity and one at zero.  The largest singular value can easily be bounded by a polynomial in $N$.  The main difficulty is controlling the least singular value.  

In order to study $\nu_{\mat M}$ it is useful to note that it is also the empirical spectral measure for the Hermitization of $\mat M$. The \emph{Hermitization} of $\mat M$ is defined to be 
\[ \mat H:=\begin{bmatrix} 
				0 &  \mat M\\
                       		\mat M^*&  0
			\end{bmatrix}. \]
For an $N \times N$ matrix, the Stieltjes transform of $\nu_{\mat{M} - z \mat{I}}$ is also the trace of the Hermitized resolvent.  That is, for $\eta \in \mathbb{C}^{+} := \{ w \in \mathbb{C} : \Im(w) > 0\}$, we have
\[ \int \frac{ 1  }{x-\eta}\nu_{\mat{M} - z \mat{I}}(dx) = \frac{1}{2N} \tr(\mat R( \mat q)),\]
where 
$$ \mat{R}(\mat q) := ( \mat{H} -   \mat{q} \otimes \mat{I}_N)^{-1}, \quad \mat{q} := \begin{bmatrix}  \eta  & z  \\ \bar{z}  & \eta  \end{bmatrix}.  $$
Here $\mat{q} \otimes \mat{I}_N$ denotes the Kronecker product of the matrix $\mat{q}$ and the identity matrix $\mat{I}_N$.  

Typically, in order to estimate the measures $\nu_{\mat{M} -z \mat{I}}$, one shows that the Stieltjes transform approximately satisfies a fixed point equation. Then one can show that this Stieltjes transform is close to the Stieltjes transform that exactly solves the fixed point equation.  Because of the dependencies between entries in the matrix $\mat Z_N$, directly computing the trace of the resolvent of the Hermitization of $\mat Z_N$ is troublesome. To circumvent this issue, in Section \ref{sec:resolvent}, instead of taking the trace of the resolvent, we instead take the partial trace and consider a $2m \times 2m$ matrix-valued Stieltjes transform. Then we show this partial trace approximately satisfies a matrix-valued fixed point equation.

In Section \ref{sec:singularvalue}, we deduce a bound for the least singular value of the matrix $\frac{1}{\sqrt{N}}(\mat{Y}_N + \mat{A}_N) - z \mat{I}$ from the known bounds on the least singular values of the individual matrices $\mat Y_{N,k} +\mat A_{N,k}$.  We finally complete the proof of Theorem \ref{thm:circular} in Section \ref{sec:complete}.  

The proof of Theorem \ref{thm:prodWigner} is very similar to the proof of Theorem \ref{thm:main}.  In fact, there are only a few places in the proof of Theorem \ref{thm:main} where the condition $|\rho_k| < 1$ is required.  We prove Theorem \ref{thm:prodWigner} in Section \ref{sec:Wigner}.  

\subsection{A remark from free probability}

The fact that the limiting distribution of the product is isotropic when the limiting distributions of the individual matrices are not might be surprising at first. Free probability, which offers a natural way to study limits of random matrices by considering joint distributions of elements from a non-commutative probability space, can shed some light on this. In free probability, the natural distribution of non-normal elements is known as the Brown measure.  For an introduction to free probability, we refer the reader to \cite{HP}; see \cite{NS} for further details about $R$-diagonal pairs as well as \cite{BL,HL} for computations of Brown measures. The distribution in Theorem \ref{thm:main} has also appeared in \cite{L}. 

A non-commutative probability space is a unital algebra $\mathcal{A}$ with a tracial state $\tau$. We say a collection of elements $a_1, \ldots, a_m$ are free if 
\[ \tau(p_1(a_{i(1)}) \ldots p_k(a_{i(k)}) ) =0 \]
whenever $p_1, \ldots, p_k$ are polynomials such that $\tau( p_j(a_{i(j)}) ) = 0$, $1 \leq j \leq k$ and $i(1) \not = i(2) \not = \ldots \not = i(k)$.  

In free probability, there are a distinguished set of elements known as $R$-diagonal elements. We refer the reader to \cite[Section 4.4]{HP} for complete details. These operators enjoy several nice properties. When they are non-singular, one such property is that their polar decomposition is $ u h$, where $u$ is a haar unitary operator, $h$ is a positive operator, and $u,h$ are free.  As a result of this decomposition, their Brown measure is isotropic. Additionally, the set of $R$-diagonal operators is closed under addition and multiplication of free elements.

In many cases, the Brown measure can be computed using the techniques of \cite{BL,HL}; however, for the purposes of this note (and due to discontinuities of the Brown measure), we will instead focus on a purely random matrix approach when computing the limiting distribution.  

We conclude this subsection by showing that the product of two elliptical elements is $R$-diagonal. We consider two elements for simplicity; however, the argument easily generalizes to the product of $m$ elliptical elements. 

First, we decompose an elliptical operator into the sum of a semicircular and circular elements, that are free from each other:
$e_i = \sqrt{\rho_i} s_i + \sqrt{1-\rho_i} c_i$.  Since the sum of free $R$-diagonal elements is again $R$-diagonal, it suffices to consider each term in the sum $(\sqrt{\rho_1} s_1 + \sqrt{1-\rho_1} c_1)(\sqrt{\rho_2} s_2 + \sqrt{1-\rho_2} c_2)$ individually and then observe that the terms are free from one another. Each term is of the form $x_1 x_2$, where $x_i$ is either semicircular or circular, with polar decomposition:
$s = a h$, $c = u h$, where $h$ is a quarter circular element, $a$ has distribution 1/2 at -1 and 1/2 at 1, and commutes with $h$, $u$ is haar unitary free from $h$. Then we consider the product:
\[ x_1 x_2 = v_1 h_1 v_2 h_2 .\]
We begin by introducing a new free haar unitary $u$. Indeed, $x_1 x_2$ has the same distribution as
\[ u v_1 h_1 u^* v_2 h_2. \]
Then $u v_1$ and $u^* v_2$ are haar unitaries, and one can check they are free from each other and $h_1$ and $h_2$. Since the product of $R$-diagonal elements remains $R$-diagonal $x_1 x_2$ is $R$-diagonal. 
Repeating this process for each term leads to the sum of free $R$-diagonal operators.


\subsection{Notation}

We use asymptotic notation (such as $O,o,\Omega$) under the assumption that $N \rightarrow \infty$.  We use $X \ll Y, Y \gg X, Y=\Omega(X)$, or $X = O(Y)$ to denote the bound $X \leq CY$ for all sufficiently large $N$ and for some constant $C$.  Notations such as $X \ll_k Y$ and $X=O_k(Y)$ mean that the hidden constant $C$ depends on another constant $k$.  We always allow the implicit constants in our asymptotic notation to depend on the integer $m$ from Theorem \ref{thm:main}; we will not denote this dependence with a subscript.  $X=o(Y)$ or $Y=\omega(X)$ means that $X/Y \rightarrow 0$ as $N \rightarrow \infty$.  

$\|\mat{M}\|$ is the spectral norm of the matrix $\mat{M}$.  $\|\mat{M}\|_2$ denotes the Hilbert-Schmidt norm of $\mat{M}$ (defined in \eqref{eq:def:hs}).  We let $\mat{I}_N$ denote the $N \times N$ identity matrix.  Often we will just write $\mat{I}$ for the identity matrix when the size can be deduced from the context.  

We write a.s., a.a., and a.e. for almost surely, Lebesgue almost all, and Lebesgue almost everywhere respectively.  We use $\sqrt{-1}$ to denote the imaginary unit and reserve $i$ as an index.  We let $\oindicator{E}$ denote the indicator function of the event $E$.

We let $C$ and $K$ denote constants that are non-random and may take on different values from one appearance to the next.  The notation $K_p$ means that the constant $K$ depends on another parameter $p$.  We always allow the constants $C$ and $K$ to depend on the integer $m$ from Theorem \ref{thm:main}; we will not denote this dependence with a subscript.  

In view of Theorem \ref{thm:main} and Assumption \ref{ass:atom}, we define the correlations $\rho_k := \E[ \xi_{k,1} \xi_{k,2} ]$ for the atom variables $\xi_{k,1}, \xi_{k,2}$.  In addition, we let $\tau > 0$ be such that
$$ \sum_{k=1}^m \left( \E|\xi_{k,1}|^{2+\tau} + \E|\xi_{k,2}|^{2+\tau} \right) < \infty. $$

\section{Proof of Theorem \ref{thm:main} } \label{sec:linear}




We begin by proving Theorem \ref{thm:main} assuming Theorem \ref{thm:circular}. The majority of the paper will then be devoted to proving Theorem \ref{thm:circular}.    

We remind the reader that the matrices $\mat{P}_N, \mat{Y}_N$, and $\mat{A}_N$ are defined in \eqref{def:PN}, \eqref{def:YN}, and \eqref{def:AN}, respectively.

\begin{proof}[Proof of Theorem \ref{thm:main}]
Let $\mat{Z}_N := \frac{1}{\sqrt{N}} (\mat{Y}_N + \mat{A}_N)$.  Then $\mat{Z}_N^m$ is a block diagonal matrix of the form
$$ \begin{bmatrix} \mat{Z}_{N,1} & & 0 \\ & \ddots & \\ 0 & & \mat{Z}_{N,m} \end{bmatrix}, $$
where $\mat{Z}_{N,1} := \mat{P}_N$ and $\mat{Z}_{N,k}$ is the matrix 
$$ N^{-m/2} (\mat{Y}_{N,k} + \mat{A}_{N,k}) \cdots (\mat{Y}_{N,m} + \mat{A}_{N,m}) (\mat{Y}_{N,1} + \mat{A}_{N,1}) \cdots (\mat{Y}_{N,k-1} + \mat{A}_{N,k-1}) $$
for $1 < k \leq m$.  

Let $f: \mathbb{C} \to \mathbb{C}$ be a bounded and continuous function.  Since each  $\mat{Z}_{N,k}$ has the same eigenvalues as $\mat{P}_N$, we have
$$ \int_{\mathbb{C}} f(z) d \mu_{\mat{P}_N}(z) = \frac{1}{N} \sum_{i=1}^n f(\lambda_i(\mat{P}_N)) = \frac{1}{mN} \sum_{i=1}^{mN} f(\lambda_i(\mat{Z}_N^m)) = \int_{\mathbb{C}} f(z^m) d \mu_{\mat{Z}_N}(z). $$
By Theorem \ref{thm:circular}, we have almost surely 
$$ \int_{\mathbb{C}} f(z^m) d \mu_{\mat{Z}_N}(z) \longrightarrow \frac{1}{\pi} \int_{\mathbb{D}} f(z^m) d^2 z  $$
as $N \to \infty$, where $\mathbb{D}$ is the unit disk in the complex plane centered at the origin and $d^2 z = d \Re(z) d \Im(z)$.  Thus, by the transformation $z \mapsto z^m$, we obtain
$$ \frac{1}{\pi} \int_{\mathbb{D}} f(z^m) d^2z = \frac{m}{\pi} \int_{\mathbb{D}} f(z) \frac{1}{m^2} |z|^{\frac{2}{m} -2} d^2z, $$
where the factor of $m$ out front of the integral corresponds to the fact that the transformation maps the complex plane $m$ times onto itself.  

Combining the computations above, we conclude that almost surely
$$  \int_{\mathbb{C}} f(z) d \mu_{\mat{P}_N}(z) \longrightarrow \frac{1}{\pi m} \int_{\mathbb{D}} f(z) |z|^{\frac{2}{m} -2} d^2z $$
as $N \to \infty$.  Since $f$ was an arbitrary bounded and continuous function, the proof of Theorem \ref{thm:main} is complete.  
\end{proof}



\section{A matrix-valued Stieltjes transform} \label{sec:resolvent}


In this section, we define a matrix-valued Stieltjes transform and introduce the relevant notation and limiting objects. Then we show that this Stieltjes transform concentrates around its expectation and estimate the error between its expectation and the limiting transform.

Here and in the sequel, we will take advantage of the following form for the inverse of a partitioned matrix (see, for instance, \cite[Section 0.7.3]{HJ}):
\begin{equation} \label{eq:schur}
	\begin{bmatrix} \mat{A} & \mat{B} \\ \mat{C} & \mat{D} \end{bmatrix}^{-1} = \begin{bmatrix} (\mat{A} - \mat{B} \mat{D}^{-1} \mat{C})^{-1} & - \mat{A}^{-1} \mat{B} (\mat{D} - \mat{C} \mat{A}^{-1} \mat{B})^{-1} \\ - (\mat{D} - \mat{C} \mat{A}^{-1} \mat{B})^{-1} \mat{C} \mat{A}^{-1} & (\mat{D} - \mat{C} \mat{A}^{-1} \mat{B})^{-1} \end{bmatrix},
\end{equation}
where $\mat{A}$ and $\mat{D}$ are square matrices.  

Set $\mat{X}_{N,k} := \frac{1}{\sqrt{N}} \mat{Y}_{N,k}$, and let $\mat{X}_N := \frac{1}{\sqrt{N}} \mat Y_N$.  Let $\mat{H}_N$ be the Hermitization of $\mat{X}_N$.  Define the resolvent
$$ \mat{R}_N(\mat q) := ( \mat{H}_N -   \mat{q} \otimes \mat{I}_N)^{-1}, $$
where 
\begin{equation} \label{def:q}
	\mat{q} := \begin{bmatrix}  \eta \mat{I}_{m} & z \mat{I}_{m} \\ \bar{z} \mat{I}_{m} & \eta \mat{I}_{m} \end{bmatrix} 
\end{equation}
for $\eta \in \mathbb{C}^{+} := \{ w \in \mathbb{C} : \Im(w) > 0\}$.  

By the Stieltjes inversion formula, $\nu_{\mat X_N - z \mat I_N}$ can be recovered from $\frac{1}{2 N m} \tr \mat R_N(\mat q)$. Because of the dependencies between matrix entries, each entry of the resolvent cannot be computed directly by Schur's Complement. One possible way to compute resolvent entries is by following the approach in \cite[Section 4.3]{OS} and use a decoupling formula to compute matrix entries. See also \cite{Nell, NgO} for computations in the elliptical case. The dependencies introduce more terms to these computations, leading to a system of equations involving diagonal entries of each block of the resolvent. These equations do not seem to admit an obvious solution. Instead we offer a matrix-valued interpretations of these equations as well as a more direct derivation of the equations.

In order to study the resolvent we will retain the block structure of $\mat H_N$ and view $2mN \times 2mN$ matrices as elements of $2m \times 2m$ matrices tensored with $N \times N$  matrices. Taking this view, $\mat R_N$ is $2m$ by $2m$ matrix with $N$ by $N$ blocks. When we wish to refer to one of these blocks (or more generally any element of a $2m \times 2m$ matrix) we will use a superscript $ab$ for the $ab^{th}$ entry. Instead of considering the full trace of $\mat R_N$, we instead take the partial trace over the $N\times N$ matrix part of the tensor product and define
$\mat{\Gamma}_N(\mat q) := (\mat I_{2m} \otimes  \frac{1}{N} \tr) \mat{R}_N(\mat q)$. That is, $\mat{\Gamma}_N(\mat q)$ is a $2m \times 2m$ matrix whose $ab^{th}$ entry is the normalized trace of the $ab^{th}$ block of $\mat{R}_N(\mat q)$.  In other words, $\mat{\Gamma}_N^{ab}(\mat q) = \frac{1}{N} \tr \mat{R}_N^{ab}(\mat q)$.  To compute this partial trace we consider $\mat R_{N;kk}$, the $2m \times 2m$ matrix whose $ab^{th}$ entry is the $(k,k)$ entry of the block $\mat R_N^{ab}$. Finally, we define the scalar
$$ a_N(\mat q) := \frac{1}{2m} \tr \mat \Gamma_N(\mat q). $$

For each $1 \leq k \leq N$, let $\mat{Y}_N^{(k)}$ denote the matrix $\mat{Y}_N$ with the $k$-th rows and $k$-th columns of $\mat{Y}_{N,1}, \ldots, \mat{Y}_{N,m}$ replaced by zeroes.  Let $\mat{H}_N^{(k)}$ be the Hermitization of $\frac{1}{\sqrt{N}} \mat{Y}^{(k)}_N$.  Define the resolvent
\begin{equation}
\label{kres}
 \mat{R}^{(k)}_N := ( \mat{H}_N^{(k)} - \mat{q} \otimes \mat{I}_N)^{-1},
 \end{equation}
and set $\mat \Gamma_N^{(k)}(\mat q) := (\mat I_{2m} \otimes  \frac{1}{N} \tr) \mat{R}_N^{(k)}(\mat q)$.  

 Let $\mat H_{N;k}^{(k)}$ be the $2m \times 2m$ matrix whose $ab^{th}$ entry is the $k^{th}$ column of the $ab^{th}$ block of $\mat H_N$, with the $k^{th}$ entry of each vector set to $0$. Note that we use a semi-colon when we refer to matrix entries or columns, in contrast to the comma, which referred to a matrix.

Later in this section we will show that $\mat{\Gamma}_N $ approximately satisfies the fixed point equation
\begin{equation} 
\label{gammadef}
\mat{\Gamma} = - ( \mat q + \mat \Sigma( \mat \Gamma ) )^{-1} 
\end{equation}
with $\mat \Sigma$ being a linear operator on $2m \times 2m$ matrices defined by:
\[ \mat \Sigma( \mat A )_{ab} = \sum_{c,d=1}^{2m} \sigma(a,c;d,b) \mat A_{cd}  \]
where $ \sigma(a,c;d,b) = N\E[H_{12}^{ac} H_{21}^{db}]$ and $\mat H^{ab}$ is the $N \times N$ matrix that is the $(a,b)^{th}$ block of the matrix $\mat H_N$, of course, the choice $(1,2)$ was arbitrary.
More concretely, we define $a'$ for any $1 \leq a \leq 2m$ to be the column index of nonzero block in the $a^{th}$ row of $\mat H_N$.  So
\[ \mat \Sigma( \mat A )_{ab} =  \mat{A}_{a'a'  } \delta_{ab}  + \rho_a \mat{A}_{a'a  } \delta_{a' b}, \]
where for $a > m$ we define $\rho_a := \rho_{a'}$.
It is important that $\mat \Sigma$ leaves diagonal entries of $\mat A$ on the diagonal and that $a' \not = a\pm m$.

To describe the limiting matrix-valued Stieltjes transform, $\mat \Gamma( \mat q)$, we first define $a(\mat q)$, the Stieltjes transform corresponding to the circular law.  That is, for each $z \in \mathbb{C}$, $a(\mat q)$ is the unique Stieltjes transform that solves the equation
\begin{equation}
\label{defa}
 a(\mat q) = \frac{a(\mat q) + \eta}{|z|^2 - (a(\mat q)+\eta)^2} 
 \end{equation}
for all $\eta \in \mathbb{C}^{+}$; see \cite[Section 3]{GTcirc}.    

Let
$$\mat{\Gamma}(\mat q) := \begin{bmatrix} -( a(\mat q) + \eta) \mat{I}_m & -z \mat{I}_m \\ -\bar{z} \mat{I}_m & -(a(\mat q) + \eta) \mat{I}_m \end{bmatrix}^{-1}; $$
additionally, \eqref{defa} implies the equality
\[\mat \Gamma(\mat q) = \begin{bmatrix}  a(\mat q) \mat{I}_{m} & \frac{z}{(a(\mat q)+\eta)^2-|z|^2} \mat{I}_{m} \\ \frac{\overline{z} }{(a(\mat q)+\eta)^2-|z|^2} \mat{I}_{m} & a(\mat q) \mat{I}_{m} \end{bmatrix}. \]

We recall that for a square matrix $\mat{M}$, the imaginary part of $\mat{M}$ is given by $\Im(\mat{M}) = \frac{1}{2 \sqrt{-1}} ( \mat{M} - \mat{M}^\ast)$.  We say $\mat{M}$ has \emph{positive imaginary part} if $\Im(\mat{M})$ is positive definite.  It was shown in \cite{HFS} that \eqref{gammadef} has one solution with positive imaginary part and is therefore a matrix-valued Stietljes transform. Furthermore, the last two equalities show that $\mat \Gamma(\mat q)$ is a solution to \eqref{gammadef}. 

A good way to see that the solution to \eqref{gammadef} is of the given form is to note that for large $\eta$, $\mat \Gamma(\mat q)$ is approximately $\mat q^{-1}$. Then by analytic continuation, the entries of $\mat q^{-1}$ that are non-zero must also be non-zero entries of $\mat \Gamma$. Finally, this ansatz for the form of the solution is applied to \eqref{gammadef} and iterated until the non-zero entries of $\mat \Gamma$ are preserved by \eqref{gammadef}. Through this process one observes that the value of each $\rho_a$ does not affect the solution.

\subsection{ Concentration }
\label{concentrationsection}
In this section we show that $\mat \Gamma_N$ concentrates around its expectation.

We introduce $\eps$-nets as a convenient way to discretize a compact set.  Let $\eps > 0$.  A set $X$ is an $\eps$-net of a set $Y$ if for any $y \in Y$, there exists $x \in X$ such that $\|x-y\| \leq \eps$.   The following estimate for the maximum size of an $\eps$-net is well-known and follows from a standard volume argument (see, for example, \cite[Lemma 3.11]{OR}).

\begin{lemma}[Lemma 3.11 from \cite{OR}] \label{lemma:epsnet}
Let $D$ be a compact subset of $\{z \in \mathbb{C} : |z| \leq M \}$.  Then $D$ admits an $\eps$-net of size at most
$$ \left( 1 + \frac{2M}{\eps} \right)^2. $$
\end{lemma}

\begin{lemma} \label{lemma:concentrate}
Let $M > 0$.  Under the assumptions of Theorem \ref{thm:main}, a.s.
$$ \sup_{|z| \leq M} \sup_{ |\eta| \leq M, \Im(\eta) \geq N^{-1/8}} \| \mat{\Gamma}_N(\mat{q}) - \E \mat{\Gamma}_N(\mat{q}) \| = O_M(N^{-1/8}). $$
\end{lemma}
\begin{proof}
By the Borel-Cantelli lemma, it suffices to show that
$$ \Prob \left( \sup_{|z| \leq M} \sup_{ |\eta| \leq M, \Im(\eta) \geq N^{-1/8}} \| \mat{\Gamma}_N(\mat{q}) - \E \mat{\Gamma}_N(\mat{q}) \| \geq C N^{-1/8} \right) \leq \frac{1}{N^2}, $$
for some constant $C>0$.  

Let $\mathcal{N}_1$ and $\mathcal{N}_2$ be $N^{-1}$-nets of $\{z \in \mathbb{C} : |z| \leq M\}$ and $\{\eta \in \mathbb{C} : |\eta|\leq M, \Im(\eta) \geq N^{-1/8} \}$ respectively. By Lemma \ref{lemma:epsnet}, 
$$ |\mathcal{N}_1| + |\mathcal{N}_2| = O_M(N^{2}). $$
Let $\mathcal{N}$ be the set of all $\mat{q}$ (defined by \eqref{def:q}) such that $z \in \mathcal{N}_1$ and $\eta \in \mathcal{N}_2$.  Hence $|\mathcal{N}| = |\mathcal{N}_1| |\mathcal{N}_2| = O_M(N^{4})$.  By the resolvent identity, 
$$ \| \mat{\Gamma}_N(\mat{q}) - \mat{\Gamma}_N(\mat{q'}) \| \leq N^{1/4} \| \mat{q} - \mat{q}' \| .$$
 Thus, by a standard $\eps$-net argument, it suffices to show that
$$ \Prob \left( \sup_{q \in \mathcal{N}} \| \mat{\Gamma}_N(\mat{q}) - \E \mat{\Gamma}_N(\mat{q}) \| \geq N^{-1/8} \right) \leq \frac{1}{N^2}. $$

By the union bound and Markov's inequality, we have, for any $p > 1$, 
\begin{align*}
	\Prob \left( \sup_{q \in \mathcal{N}} \| \mat{\Gamma}_N(\mat{q}) - \E \mat{\Gamma}_N(\mat{q}) \| \geq N^{-1/8} \right) &\leq \sum_{\mat{q} \in \mathcal{N}} \Prob \left( \| \mat{\Gamma}_N(\mat{q}) - \E \mat{\Gamma}_N(\mat{q}) \| \geq N^{-1/8} \right) \\
		&\leq \sum_{\mat{q} \in \mathcal{N}} N^{p/8} \E \| \mat{\Gamma}_N(\mat{q}) - \E \mat{\Gamma}_N(\mat{q}) \|^p.
\end{align*}
Therefore, it will suffice to show that for some $p > 1$ sufficiently large, there exists a constant $K_p > 0$ (depending only on $p$), such that 
$$ \E \| \mat{\Gamma}_N(\mat{q}) - \E \mat{\Gamma}_N(\mat{q}) \|^p \leq \frac{K_p}{N^{3p/8}} $$
for any $\mat{q} \in \mathcal{N}$.  

In fact, since $\mat{\Gamma}_N(\mat{q})$ is a $2m \times 2m$ matrix, we will show that, for every $p > 1$, 
\begin{equation} \label{eq:gammabshow}
	\E | \mat{\Gamma}^{ab}_N(\mat{q}) - \E \mat{\Gamma}^{ab}_N(\mat{q}) |^p \leq \frac{K_p}{N^{3p/8}} 
\end{equation}
for any $1 \leq a,b \leq 2m$ and $\mat{q} \in \mathcal{N}$.  

Fix $1 \leq a,b \leq 2m$.  Let $\E_k$ denote the conditional expectation with respect to the first $k$ rows and $k$ columns of each matrix $\mat{Y}_{N,1}, \ldots, \mat{Y}_{N,m}$. 


We now rewrite $\mat{\Gamma}^{ab}_N(\mat{q}) - \E \mat{\Gamma}^{ab}_N(\mat{q})$ as a martingale difference sequence.  Indeed,
\begin{align*}
	\mat{\Gamma}^{ab}_N(\mat{q}) - \E \mat{\Gamma}^{ab}_N(\mat{q}) &= \sum_{k=1}^N (\E_k - \E_{k-1}) \mat{\Gamma}^{ab}_N(\mat{q}) \\
		&= \sum_{k=1}^N (\E_k - \E_{k-1}) \left( \mat{\Gamma}^{ab}_N(\mat{q}) - \mat{\Gamma}^{(k) ab}_N(\mat{q}) \right).  
\end{align*}
Since $\mat{Y}_N - \mat{Y}_N^{(k)}$ is at most rank $2m$, the resolvent identity implies that $\mat{R}_N(\mat{q}) - \mat{R}^{(k)}_N(\mat{q})$ is at most rank $4m$ and $\| \mat{R}_N(\mat{q}) - \mat{R}^{(k)}_N(\mat{q}) \| \leq 8m|\Im(\eta)|^{-1}$. This then gives the bound
\begin{equation}
\label{rescomp}
 \left|\mat{\Gamma}^{ab}_N(\mat{q}) - \mat{\Gamma}^{(k) ab}_N(\mat{q})\right| \leq \frac{8m}{N|\Im(\eta)|} \leq 8m N^{-7/8} 
 \end{equation}
for any $\mat{q} \in \mathcal{N}$.  Thus, by the Burkholder inequality \cite{Bmart} (see for example \cite[Lemma 2.12]{BSbook} for a complex-valued version of the Burkholder inequality), for any $p > 1$, 
\begin{align}
\label{pthest}
	\E \left| \mat{\Gamma}^{ab}_N(\mat{q}) - \E \mat{\Gamma}^{ab}_N(\mat{q}) \right|^p &\leq K_p \E \left( \sum_{k=1}^N \left| \mat{\Gamma}^{ab}_N(\mat{q}) - \mat{\Gamma}^{(k) ab}_N(\mat{q}) \right|^2 \right)^{p/2} \notag \\
		&\leq K_p N^{-p/2} |\Im(\eta)|^{-p} \leq K_p N^{-3p/8}
\end{align}
for any $\mat{q} \in \mathcal{N}$.  This verifies \eqref{eq:gammabshow}, and hence the proof of the lemma is complete.  
\end{proof}

\subsection{Estimate of the expectation}

We now estimate the difference between $\E[\mat \Gamma_N(\mat q)]$ and $\mat \Gamma(\mat q)$. In Section \ref{sec:truncation}, we show that it suffices to assume that the entries of $\mat Y_{N,k}$ are truncated and that the diagonal entries are zero (as assumed in Lemma \ref{approxthm} below). Furthermore, the matrices can be renormalized so that the variance of the non-zero entries of $\mat Y_{N,k}$ is one. The correlations between the $(i,j)$-entry and $(j,i)$-entry might change and be $N$ dependent, but our results are independent of their values.  In Section \ref{sec:truncation}, we give a truncation argument and use hats and superscripts to denote this truncation.  For notational convenience, we omit the hats and subscripts for the remainder of this section.

\begin{lemma}[Self-consistent equation]
\label{approxthm}
Let $M > 0$. Under the assumptions of Theorem \ref{thm:main}, with the additional assumptions that $\zeta_k = 0$ and $\xi_{k,1}, \xi_{k,2}$ are bounded by $N^{\delta}$, for some $0 < \delta < 1/2$, we have
$$ \sup_{|z| \leq M, |\eta| \leq M} \| \E \mat{\Gamma}_N(\mat{q}) - \mat{\Gamma}(\mat{q}) \| = O_{M} \left( N^{-1/2} |\Im(\eta)|^{-5} \right). $$
\end{lemma}

We are interested in the resolvent evaluated at $\mat q$, but it can be defined at any $2m \times 2m$ matrix with positive imaginary part. 

Indeed, let  $\tilde{ \mat q}$ be a $2m \times 2m$ matrix whose imaginary part is a positive definite matrix. We then expand our definition of the resolvent to $\mat R_N(\tilde{ \mat q}) = ( \mat H_N - \tilde{ \mat 	q} \otimes \mat I_N )^{-1}$. We still have the trivial bound $\| \mat R_N(\tilde{ \mat q}) \| \leq \|\Im(\tilde{ \mat q})^{-1}\|$; see for instance \cite[Lemma 3.1]{HT}.

We begin by computing the expectation of the $2m \times 2m$ matrix formed by taking a diagonal entry from each block of the resolvent. By exchangeability, this is also the expectation of the partial trace. 

By \eqref{eq:schur}, we have
\begin{align*}
\mat R_{N;11}(\tilde{ \mat q}) = -( \tilde{ \mat q}  +\mat H_{N;1}^{(1)*} \mat R_N^{(1)}(\tilde{ \mat q}) \mat H_{N;1}^{(1)} )^{-1}.
\end{align*}

\begin{lemma}
\label{error}
Let $\mat H_N$ be as in Lemma \ref{approxthm}. Let $\mat R_N(\tilde{ \mat q}) = ( \mat H_N - \tilde{ \mat q} \otimes \mat I_N )^{-1}$ and $\mat \Gamma_N(\tilde{ \mat q})= (\mat I_{2m} \otimes  \frac{1}{N} \tr)  \mat R_N(\tilde{ \mat q}) $, then
 \[\E[  \|  \mat\Sigma( \E[ \mat \Gamma_N(\tilde{ \mat q}) ] )-  \mat H_{N;1}^{(1)*} \mat R_N^{(1)}(\tilde{ \mat q}) \mat H_{N;1}^{(1)} \|]   = O( N^{-1/2} \| \Im(\mat{ \tilde q})^{-1}\| ) \]
\end{lemma}

\begin{proof} 
We divide our estimate into three parts, the first two bounds follow from estimates in Section \ref{concentrationsection}, replacing the bound $\| \mat R_N(\mat q)\| \leq |\Im(\eta)|^{-1}$ with $\| \mat R_N(\tilde{ \mat q}) \| \leq  \| \Im(\tilde{ \mat q})^{-1} \|$.
From \eqref{rescomp}, we have the deterministic bound
\[\| \mat \Gamma_N(\tilde{ \mat q})- \mat \Gamma_N^{(1)}(\tilde{ \mat q}) \| \leq K N^{-1}  \|\Im(\tilde{ \mat q})^{-1}\| .\] 
Then using the concentration inequality \eqref{pthest} along with Jensen's inequality,
\[\E[\| \mat \Gamma_N - \E[\mat \Gamma_N] \|] \leq  K N^{-1/2} \|\Im(\tilde{ \mat q})^{-1}\| .\] 
Finally, we note that
\begin{align*}
\E_{1}[ \mat H_{N;1}^{(1)*} \mat R_N^{(1)} \mat H_{N;1}^{(1)} ] = \mat \Sigma( \mat \Gamma_N^{(1)} ) -(N \mat{ \tilde q} )^{-1}.
\end{align*}
Recall that $\E_{1}$ denotes conditional expectation with respect to the first row and column of each matrix $Y_{N,1}, \ldots, Y_{N,m}$.
The last term appears because $\mat R_{N,11}^{(1)} = \mat{\tilde{q}}^{-1}$ but $\mat H_{N;11}^{(1)} =\mat  0$.
Then a direct computation, using the truncation on the matrix entries and the trivial bound $\|\mat R^{(1)} \| \leq \|\Im(\mat{ \tilde{ q}})^{-1}\|$, yields the bound
\begin{align*}
\E[|(\mat H_{N;1}^{(1)*} \mat R_N^{(1)} \mat H_{N;1}^{(1)} - \mat \Sigma( \mat \Gamma_N^{(1)} ))^{ab}|^2] \leq 
\frac{K N^{2\delta}}{N^2} \|\Im(\tilde{ \mat q})^{-2}\|.
\end{align*}


\end{proof}

Let $\mat \epsilon_{N} :=  - \mat\Sigma( \E[ \mat \Gamma_N ] )+  \mat H_{N;1}^{(1)*} \mat R_N^{(1)} \mat H_{N;1}^{(1)} $ be the error term in the previous lemma. 
We now begin the proof of Lemma \ref{approxthm}.

\begin{proof}[Proof of Lemma \ref{approxthm}]
We first observe that if $|\Im \eta| \leq 4^{1/4} N^{-1/8}$ then 
\[ \| \E \mat{\Gamma}_N(\mat{q}) - \mat{\Gamma}(\mat{q}) \|  \leq 2 \Im(\eta)^{-1} = O( N^{-1/2} |\Im(\eta)^{-5}|),   \]
so we will assume $|\Im(\eta)| > 4^{1/4} N^{-1/8}$. This condition will be useful when bounding $\epsilon_N'$ (defined below). 

From Lemma \ref{error}, $\E[\mat \Gamma_N(\mat q)]$ approximately satisfies the defining equation for $\mat \Gamma(\mat q)$. We will now show we can deform $\E[\mat \Gamma_N(\mat q)]$ so that it exactly satisfies the defining equation for $\mat \Gamma(\mat q)$ and then finish by bounding the difference between $\E[\mat \Gamma_N(\mat q)]$ and the deformed version.  Indeed,
\begin{align}
\label{expeq}
\E[\mat \Gamma_N(\mat{q}) ]  &= \E[\mat R_{N;11}(\mat{q})]= - \E[( \mat q  + \mat \Sigma( \E[ \mat  \Gamma_N(\mat{q})] ) + \mat \epsilon_N)^{-1} ] \notag\\
&= - (\mat  q  + \mat \Sigma(\E[\mat  \Gamma_N(\mat{q}) ] ))^{-1} \E[1 +  \mat \epsilon_N     (\mat q  + \mat \Sigma(\E[ \mat  \Gamma_N(\mat{q}) ])+\mat \epsilon_N )^{-1}].  
\end{align}
Lemma \ref{error} gives the following bound on the second term
which then implies
\[\| (\mat  q  + \mat \Sigma(\E[\mat  \Gamma_N(\mat q) ] ))^{-1}\| \leq 2 |\Im(\eta)|^{-1}. \]

Returning to \eqref{expeq}, we obtain
\begin{align*}
\E[\mat \Gamma_N (\mat q)]   &= - (\mat  q  + \mat \Sigma(\E[\mat  \Gamma_N(\mat q)] ))^{-1}  + \epsilon_N',
\end{align*}
where 
\[ \epsilon_N' := - \E[    (\mat q  + \mat \Sigma(\E[ \mat  \Gamma_N(\mat q)] ))^{-1}   \mat \epsilon_N   \mat R_{N;11}(\mat q)]. \]
The previous estimates show 
\[\| \epsilon_N' \|\leq 2 |\Im(\eta)|^{-2} \E[\| \mat  \epsilon_N \|] \leq 2 |\Im(\eta)|^{-3} N^{-1/2} \leq \frac{|\Im(\eta)| }{2}. \]
The last inequality uses our assumption that $\Im(\eta) > 4^{1/4} N^{-1/8}$.

Let $\mat q_N := \mat q +  \mat \Sigma(\mat \epsilon_N')$. First note that $\| \Im(\mat q_N)\| \geq \Im(\eta)/2$, then observe that:
\[\E[ \mat \Gamma_N(\mat q)] - \epsilon_N' = - (\mat q_N + \Sigma( \E[ \mat \Gamma_N(\mat q)]   -\mat \epsilon_N')  )^{-1};\]
so $\E[ \mat \Gamma_N(\mat q)] - \epsilon_N'$ satisfies the same equation as $\mat \Gamma(\mat  q_N)$, then by uniqueness of the solution 
\[\E[ \mat \Gamma_N(\mat q)] - \mat \epsilon_N'=\mat \Gamma(\mat q_N) .\]
%
%

%
Finally, we conclude that
\[ \E[\mat \Gamma_N(\mat q)] - \mat \Gamma(\mat q) =  \mat \Gamma(\mat q_N)) + \epsilon_N'  - \mat \Gamma(\mat q). \]
Since the right-hand side is bounded uniformly in norm by a universal constant times $ \|\epsilon_N'\| (1+|\Im(\eta)|^{-2})= O(|\Im(\eta)|^{-5}N^{-1/2})$, the proof is complete.
\end{proof}

\section{Least singular value bound} \label{sec:singularvalue}

This section is devoted to bounding the least singular value of the matrix $\frac{1}{\sqrt{N}}( \mat{Y}_N + \mat{A}_N) - z \mat{I}$.  In particular, we will prove the following theorem.  

\begin{theorem}[Least singular value bound] \label{thm:least-sing-value}
Under the assumptions of Theorem \ref{thm:main}, there exists $A >0$ such that for almost every $z \in \mathbb{C}$, almost surely
$$ \lim_{N \to \infty} \indicator{\sigma_{mN}( N^{-1/2} (\mat{Y}_N + \mat{A}_N) - z\mat{I}) \leq N^{-A}} = 0. $$
\end{theorem}

In order to prove Theorem \ref{thm:least-sing-value}, we will need the following bound from \cite{NgO}.  

\begin{theorem}[Bound on the least singular value for perturbed random matrices] \label{thm:ellipticsv}
For each $1 \leq k \leq m$, let $\mat{F}_{N,k}$ be a $N \times N$ complex deterministic matrix whose entries are bound in absolute value by $N^{\alpha}$.  Then, under the assumptions of Theorem \ref{thm:main}, for any $B > 0$, there exists $A>0$ (depending on $B, \alpha$, and $\rho_k:= \E[\xi_{k,1} \xi_{k,2}]$ for $k=1,\ldots, m$) such that
$$ \sup_{1 \leq k \leq m} \Prob \left( \sigma_N(\mat{Y}_{N,k} + \mat{F}_{N,k}) \leq N^{-A} \right) = O(N^{-B}). $$
\end{theorem}

\begin{remark} \label{rem:lsv}
Theorem \ref{thm:ellipticsv} requires item \eqref{item:rhoass} from Assumption \ref{ass:atom}.  In fact, this is the only place in the note where this assumption is required.  In the case $\rho_k=1$, $\mat{Y}_{N,k}$ is a real symmetric Wigner matrix.  Thus, one would need to control the least singular value of $\mat{Y}_{N,k} + \mat{F}_{N,k}$.  A bound was obtained in \cite{NgLSV} when $\mat{F}_{N,k}$ is also symmetric.  The case when $\mat{F}_{N,k}$ is arbitrary appears more difficult.  See Theorem \ref{thm:least-sing-valueWigner} for more details.  The case $\rho_k = -1$ is more delicate since a skew-symmetric matrix of odd degree is always singular.  
\end{remark}

We now verify Theorem \ref{thm:least-sing-value}.  

\begin{proof}[Proof of Theorem \ref{thm:least-sing-value}]
Let $A$ be a large positive constant to be chosen later.  By the Borel-Cantelli lemma, it suffices to show that, for almost every $z \in \mathbb{C}$, 
$$ \Prob \left( \sigma_{mN} \left( N^{-1/2}(\mat{Y}_N + \mat{A}_N) - z\mat{I} \right) \leq N^{-A} \right) = O(N^{-2}). $$
In other words, it suffices to show that
$$ \Prob \left( \left\| \left( \frac{1}{\sqrt{N}} (\mat{Y}_N + \mat{A}_N) - z\mat{I}\right)^{-1} \right\| \geq N^{A} \right) = O(N^{-2}). $$

Let 
$$ \mat{M}_N := \left( \frac{1}{\sqrt{N}} (\mat{Y}_N + \mat{A}_N) - z\mat{I} \right)^{-1}. $$  
We let $\mat{M}_N^{ab}$ denote the $(a,b)$-th $N \times N$ block of $\mat{M}_N$, for $a,b \in \{1,\ldots,m\}$.  Thus, we have
$$ \Prob \left( \| \mat{M}_N \| \geq N^{A} \right) \leq \Prob \left( \text{there exists } a,b \in \{1, \ldots, m\} \text{ with } \| \mat{M}_N^{ab} \| \geq \frac{1}{m^2} N^{A} \right). $$
Therefore, by the union bound, it suffices to show that
$$ \Prob \left( \| \mat{M}^{ab}_N \| \geq N^{A-1} \right) = O(N^{-2}) $$
for all $a,b \in \{1, \ldots, m\}$.  

Fix $a,b \in \{1, \ldots, m\}$, and define 
$$ \mat{V}_{N,k} := \frac{1}{\sqrt{N}} (\mat{Y}_{N,k} + \mat{A}_{N,k}) $$
for each $1 \leq k \leq m$.  By computing the inverse of a partitioned matrix (see, for instance, \cite[Section 0.7.3]{HJ}), we observe that
$$ \mat{M}_N^{ab} = z^\kappa \mat{V}_{N,j_1} \cdots \mat{V}_{N,j_l} \left( \mat{V}_{N,i_1} \cdots \mat{V}_{N,i_q} - z^r \right)^{-1}, $$
where $\kappa, l, q, r$ are non-negative integers no larger than $m$; here $\kappa, l, q, r j_1, \ldots, j_l, i_1, \ldots, i_q$ depend only on $a,b,m$, and the indicies $i_1, \ldots, i_q$ are distinct.   Thus, we obtain
\begin{align*}
	\Prob ( \| \mat{M}_N^{ab} \| \geq N^{A-1} ) &\leq \Prob( |z|^\kappa \| \mat{V}_{N,j_1} \cdots \mat{V}_{N,j_l} \| \geq N^{(A-1)/2}) \\
		&\qquad+ \Prob( \| ( \mat{V}_{N,i_1} \cdots \mat{V}_{N,i_q} - z^r \mat{I})^{-1} \| \geq N^{(A-1)/2} ).
\end{align*}
We will bound each of the terms on the right-hand side separately.  

For the first term, by taking $A$ sufficiently large and applying Markov's inequality, we have
\begin{align*}
	 \Prob( |z|^\kappa \| \mat{V}_{N,j_1} \cdots \mat{V}_{N,j_l} \| \geq N^{(A-1)/2}) &\leq \sum_{s=1}^l \Prob( \| \mat{V}_{N,j_s} \| \geq N^{(A-1)/4m}) \\
	 	&\leq \sum_{s=1}^l \frac{2}{N^{(A-1)/2m}} \frac{1}{N} \left( \E \| \mat{Y}_{N,j_s} \|^2_2 + \| \mat{A}_{N,k} \|_2^2 \right) \\
		&= O(N^{-2}).
\end{align*}
Here we used \eqref{eq:Aassump} to deduce that
$$ \max_{1 \leq k \leq m}\left( \E \| \mat{Y}_{N,j_s} \|^2_2 + \| \mat{A}_{N,k} \|_2^2 \right) = O(N^2). $$
For the second term, we observe that
$$  (\mat{V}_{N,i_1} \cdots \mat{V}_{N,i_q} - z^r \mat{I})^{-1} = \mat{V}_{N,i_q}^{-1} \cdots \mat{V}_{N,i_2}^{-1} ( \mat{V}_{N,i_1} - z^r \mat{V}_{N,i_q}^{-1} \cdots \mat{V}_{N,i_2}^{-1})^{-1}. $$
Thus, 
\begin{align*}
	\Prob &( \| ( \mat{V}_{N,i_1} \cdots \mat{V}_{N,i_q} - z^r \mat{I})^{-1} \| \geq N^{(A-1)/2} ) \\
	& \leq  \Prob (\| \mat{V}_{N,i_q}^{-1} \cdots \mat{V}_{N,i_2}^{-1} \| \geq N^{(A-1)/4}) \\
	& \qquad + \Prob( \| ( \mat{V}_{N,i_1} - z^r \mat{V}_{N,i_q}^{-1} \cdots \mat{V}_{N,i_2}^{-1})^{-1} \| \geq N^{(A-1)/4} ).
\end{align*}
Let $A'$ be a large positive constant.  By taking $A'$ sufficiently large, Theorem \ref{thm:ellipticsv} implies that
\begin{equation} \label{eq:norminverseindiv}
	\Prob \left(\| \mat{V}_{N,i_q}^{-1} \cdots \mat{V}_{N,i_2}^{-1} \| \geq N^{A'}\right) \leq \sum_{s=2}^q \Prob \left(\| \mat{V}_{N,i_s}^{-1} \| \geq N^{A'/m}\right) = O(N^{-2}). 
\end{equation}
Therefore, it suffices to show that
\begin{equation} \label{eq:inversebnd}
	\Prob( \| ( \mat{V}_{N,i_1} - z^r \mat{V}_{N,i_q}^{-1} \cdots \mat{V}_{N,i_2}^{-1})^{-1} \| \geq N^{(A-1)/4} ) = O(N^{-2}).
\end{equation}

Define the event 
$$ \Omega_N := \left\{ \| z^r \mat{V}_{N,i_q}^{-1} \cdots \mat{V}_{N,i_2}^{-1} \| \leq N^{A'+1} \right\}. $$
Then, by \eqref{eq:norminverseindiv}, 
\begin{align} \label{eq:normcomp}
	\Prob(\Omega_N^C) = O(N^{-2}). 
\end{align}
We now exploit the fact that $\mat{Y}_{N,i_1}, \ldots, \mat{Y}_{N,i_q}$ are independent.  Indeed, by freezing the matrices $\mat{Y}_{N,i_2}, \ldots, \mat{Y}_{N,i_q}$ and conditioning on $\Omega_N$, we apply Theorem \ref{thm:ellipticsv} and obtain 
$$ \Prob(  \| ( \mat{V}_{N,i_1} - z^r \mat{V}_{N,i_q}^{-1} \cdots \mat{V}_{N,i_2}^{-1})^{-1} \| \geq N^{(A-1)/4} \mid \Omega_N) = O(N^{-2}) $$
for $A$ sufficiently large.  Combining the bound above with \eqref{eq:normcomp} yields \eqref{eq:inversebnd}, and the proof is complete.  
\end{proof}

\section{Completing the argument} \label{sec:complete}

In this section, we apply the results of the previous sections to complete the proof of Theorem \ref{thm:circular}.  

\subsection{Truncation} \label{sec:truncation}

In order to apply Lemma \ref{approxthm}, we need to consider elliptic random matrices whose entries are bounded by $N^{\delta}$.  To this end, we present a number of standard truncation results below.    The proofs of these results can be found in Appendix \ref{sec:truncationproof}.

Let $(\xi_1, \xi_2)$ be a random vector in $\mathbb{R}^2$.  Let $\delta > 0$.  Set
$$ \tilde{\xi}_i^{(N)} := \xi_i \indicator{|\xi_i| \leq N^\delta} - \E [\xi_i \indicator{|\xi_i| \leq N^\delta} ] $$
and
$$ \hat{\xi}_i^{(N)} := \frac{ \tilde{\xi}_i^{(N)} }{ \sqrt{ \var( \tilde{\xi}_i^{(N)} ) } } $$
for $i = 1,2$.  We verify the following standard truncation result in Appendix \ref{sec:truncationproof}.  

\begin{lemma}[Truncation] \label{lemma:truncation}
Let $(\xi_1, \xi_2)$ be a random vector in $\mathbb{R}^2$, where $\xi_1, \xi_2$ each have mean zero, unit variance, and satisfy 
\begin{equation} \label{eq:M2tau}
	M_{2+\tau} := \E|\xi_1|^{2+\tau} + \E|\xi_2|^{2+\tau} < \infty 
\end{equation}
for some $\tau > 0$.  Set $\rho := \E[\xi_1 \xi_2]$.  Let $\delta > 0$.  Then there exists $N_0$ (depending only on $\delta, \tau$, and $M_{2+\tau}$) such that the following properties hold for all $N \geq N_0$.
\begin{enumerate}[(i)]
\item \label{item:meanvar} For $i=1,2$, $\hat{\xi}_i^{(N)}$ has mean zero, unit variance, and is a.s. bounded in magnitude by $4N^\delta$.  
\item \label{item:varbnd} For $i=1,2$, 
$$ \left| 1 - \var(\tilde{\xi}_i^{(N)}) \right| \leq 2 \frac{M_{2+\tau}}{N^{\delta \tau}}. $$
\item \label{item:rho} For $\hat{\rho}^{(N)} := \E[ \hat{\xi}_1^{(N)} \hat{\xi}_2^{(N)}]$, we have
$$ |\hat{\rho}^{(N)} - \rho| \leq 13 \frac{M_{2+\tau}}{N^{\delta \tau/2}}. $$
\end{enumerate}
\end{lemma}

We now define the truncated matrices $\tilde{\mat{Y}}_N$ and $\hat{\mat{Y}}_N$.  For each $k=1,\ldots,m$, define the $N \times N$ matrices $\tilde{\mat{Y}}_{N,k}$ and $\hat{\mat{Y}}_{N,k}$ with entries 
$$ (\tilde{\mat{Y}}_{N,k})_{ij} := \left\{
	\begin{array}{lr}
       	 (\mat{Y}_{N,k})_{ij} \indicator{|(\mat{Y}_{N,k})_{ij}| \leq N^{\delta}} - \E [ (\mat{Y}_{N,k})_{ij} \indicator{|(\mat{Y}_{N,k})_{ij}| \leq N^{\delta}} ] , & \text{for } i \neq j \\
       	0, & \text{for } i = j
     	\end{array}
   	\right. $$
and
$$ (\hat{\mat{Y}}_{N,k})_{ij} := \left\{
	\begin{array}{lr}
       	 \frac{ (\tilde{\mat{Y}}_{N,k})_{ij} }{ \sqrt{ \var( (\tilde{\mat{Y}}_{N,k})_{ij} ) } }, & \text{for } i \neq j \\
       	0, & \text{for } i = j.
     	\end{array}
   	\right. $$
Define the $mN \times mN$ block matrices
$$ \tilde{\mat{Y}}_N := \begin{bmatrix} 
				0 &  \tilde{\mat{Y}}_{N,1} &     &            & 0       \\
                         		0 & 0    & \tilde{\mat{Y}}_{N,2} &            & 0        \\      
                           		&      & \ddots & \ddots     &         \\
                         		0 &      &     &          0 & \tilde{\mat{Y}}_{N,m-1} \\
                       		\tilde{\mat{Y}}_{N,m} &      &     &            &  0     
				\end{bmatrix} $$
and 
$$ \hat{\mat{Y}}_N := \begin{bmatrix} 
				0 &  \hat{\mat{Y}}_{N,1} &     &            & 0       \\
                         		0 & 0    & \hat{\mat{Y}}_{N,2} &            & 0        \\      
                           		&      & \ddots & \ddots     &         \\
                         		0 &      &     &          0 & \hat{\mat{Y}}_{N,m-1} \\
                       		\hat{\mat{Y}}_{N,m} &      &     &            &  0      
				\end{bmatrix}. $$
				
We will also need the following lemma, whose proof is presented in Appendix \ref{sec:truncationproof}. 

\begin{lemma}[Law of large numbers] \label{lemma:lln}
Under the assumptions of Theorem \ref{thm:main}, for $\delta \tau < 1$, the following properties hold almost surely:
\begin{align}
	\limsup_{N \to \infty} \frac{1}{N^2} \| \mat{Y}_N \|_2^2 < \infty, \label{eq:limsupYN} \\
	\limsup_{N \to \infty} \frac{1}{N^2} \| \hat{\mat{Y}}_N \|_2^2 < \infty, \label{eq:limsuphat} \\
	\lim_{N \to \infty} \frac{N^{\delta \tau}}{N^2} \| \mat{Y}_N - \hat{\mat{Y}}_N \|_2^2 = 0. \label{eq:limdiff}
\end{align}
\end{lemma}

Recall that $\mat{H}_N$ is the Hermitization of $\frac{1}{\sqrt{N}} \mat{Y}_N$.  Let $\hat{\mat{H}}_N$ be the Hermitization of the matrix $\frac{1}{\sqrt{N}} \hat{\mat{Y}}_N$.  Let $F^{(N)}_z$ be the ESD of 
$$ \mat{H}_N - \begin{bmatrix} \mat{0} & z \mat{I}_m \\ \bar{z} \mat{I}_m & \mat{0} \end{bmatrix}  \otimes \mat{I}_N $$
for $z \in \mathbb{C}$.  Similarly, let $\hat{F}^{(N)}_z$ be the ESD of 
$$ \hat{\mat{H}}_N - \begin{bmatrix} \mat{0} & z \mat{I}_m \\ \bar{z} \mat{I}_m & \mat{0} \end{bmatrix}  \otimes \mat{I}_N . $$

For two cumulative distribution functions $F$ and $G$, we define the \emph{Levy distance} 
\begin{equation} \label{eq:def:levy}
	L(F, G) := \inf \{ \eps > 0 : F(x - \eps) - \eps \leq G(x) \leq F(x + \eps) + \eps \text{ for all } x \in \mathbb{R} \}.
\end{equation}

It follows that convergence in the metric $L$ implies convergence in distribution.  In fact, in the following lemma, we show that $\hat{F}^{(N)}_z$ approximates $F^{(N)}_z$ in Levy distance.  

\begin{lemma} \label{lemma:truncn}
Under the assumptions of Theorem \ref{thm:main}, for $\delta \tau < 1$, we have a.s. 
$$ \sup_{z \in \mathbb{C}} L(F^{(N)}_z, \hat{F}^{(N)}_z) = o(N^{\delta \tau/3}). $$
\end{lemma}
\begin{proof}
By \cite[Corollary A.41]{BSbook} and Lemma \ref{lemma:lln}, we have a.s.
\begin{align*}
	\limsup_{N \to \infty} \sup_{z \in \mathbb{C}} N^{\delta \tau} L^3(F^{(N)}_z, \hat{F}^{(N)}_z) & \leq \limsup_{N \to \infty} \frac{N^{\delta \tau}}{2mN} \| \mat{H}_N - \hat{\mat{H}}_N \|_2^2 \\
		&\leq \limsup_{N \to \infty} \frac{N^{\delta \tau}}{mN^2} \| \mat{Y}_N - \hat{\mat{Y}}_N \|_2^2 \\
		&= 0.
\end{align*}
Thus, we conclude that a.s. $\sup_{z \in \mathbb{C}} L(F^{(N)}_z, \hat{F}^{(N)}_z) = o(N^{\delta \tau/3})$.
\end{proof}

\subsection{Proof of Theorem \ref{thm:circular}}

This section is devoted to Theorem \ref{thm:circular}. 
We begin with a lemma, alluded to in the introduction, that allows us to connect the distribution of a non-Hermitian matrix to that of a family of Hermitian matrices.
Lemma \ref{lemma:girko} follows from \cite[Lemma 11.2]{BSbook} and is based on Girko's original observation \cite{G1,G2}.  The lemma has appeared in a number of different forms; for example, see \cite[Lemma 4.3]{BC} and \cite{GK}.  

\begin{lemma}[Lemma 11.2 from \cite{BSbook}] \label{lemma:girko}
Let $\mat{M}$ be a $N \times N$ matrix.  For any $uv \neq 0$, we have
\begin{align*}
	&\iint e^{\sqrt{-1} ux + \sqrt{-1} vy} F^{\mat{M}}(dx, dy) \\
	&\qquad \qquad= \frac{u^2 + v^2}{4 \sqrt{-1} u \pi} \iint \frac{ \partial }{\partial s} \left[ \int_{0}^\infty \ln |x|^2 \nu_{\mat{M} - z\mat{I}}(dx) \right] e^{\sqrt{-1} us + \sqrt{-1} vt} dt ds, 
\end{align*}
where $z = s + \sqrt{-1} t$.  
\end{lemma}

For a square matrix $\mat{M}$, we define the function
\begin{equation} \label{eq:def:gMn}
	g_{\mat{M}}(s,t) := \frac{\partial}{\partial s} \int_{0}^\infty \log |x|^2 \nu_{\mat{M}-z\mat{I}}(dx),
\end{equation}
where $z = s + \sqrt{-1} t$.  We also define
\begin{equation} \label{eq:def:g}
	g(s,t) := \left\{
		\begin{array}{ll}
		\frac{2s}{s^2 + t^2}, & \text{if } s^2 + t^2 >1\\
		2s, & \text{otherwise}
	\end{array}
   	\right. .
\end{equation}

We will make use of the following result from \cite{BSbook}.
\begin{lemma}[Lemma 11.5 from \cite{BSbook}] \label{lemma:BS:115}
For all $uv \neq 0$, we have
$$ \frac{1}{\pi} \iint_{x^2 + y^2 \leq 1} e^{\sqrt{-1} ux + \sqrt{-1} vy} dx dy = \frac{u^2 + v^2}{4 \sqrt{-1} u \pi} \int \left[ \int g(s,t) e^{\sqrt{-1}us + \sqrt{-1} vt } dt \right] ds. $$
\end{lemma}

We will also need the following lemma, which summarizes some of the results from \cite[Chapter 11]{BSbook}, \cite{BC}, and \cite[Section 3]{GTcirc} (in particular, see \cite[Remark 3.1]{GTcirc}).
\begin{lemma} \label{lemma:nuz}
For each $z \in \mathbb{C}$, there exists a probability measure $\nu_z$ on the real line such that the following properties hold.
\begin{enumerate}[(i)]
\item $g(s,t) = \frac{\partial}{\partial s} \int_{0}^\infty \log|x|^2 d\nu_z(dx)$, where $z := s + \sqrt{-1}t$.
\item For each $z \in \mathbb{C}$, $\nu_z$ has density $\rho_z$ with 
$$ \sup_{z \in \mathbb{C}} \sup_{x \in \mathbb{R}} |\rho_z(x)| \leq 1. $$
\item For any $M > 0$, there exists $\beta > 0$ such that $\nu_z$ is supported inside $[-\beta, \beta]$ for all $z \in \mathbb{C}$ with $|z| \leq M$.  
\item $a(\mat{q})$ (defined in \eqref{defa}) is the Stieltjes transform of $\nu_z$.  That is,
$$ a(\mat{q}) = \int \frac{1}{x - \eta} d \nu_z(dx) = \int \frac{ \rho_z(x)}{x - \eta} dx. $$
Recall that the matrix $\mat{q}$ is a function of both $\eta$ and $z$.
\end{enumerate}
\end{lemma} 

Let $\mat{Y}_N$ and $\mat{A}_N$ satisfy the assumptions of Theorem \ref{thm:main}.  By Lemma \ref{lemma:girko} and Lemma \ref{lemma:BS:115}, in order to prove Theorem \ref{thm:circular}, it suffices to show that a.s.
$$ \iint \left[g_{\frac{1}{\sqrt{N}}(\mat{Y}_N + \mat{A}_N)}(s,t) - g(s,t)\right] e^{\sqrt{-1} us + \sqrt{-1} vt} dt ds \longrightarrow 0 $$
as $N \rightarrow \infty$.  

We now claim that a.s.
\begin{equation} \label{eq:hsbnd}
	\frac{1}{N} \left \| \frac{1}{\sqrt{N}} (\mat{Y}_N + \mat{A}_N) \right\|_2^2 = O(1). 
\end{equation}
Indeed, by the triangle inequality
$$ \frac{1}{N} \left \| \frac{1}{\sqrt{N}} (\mat{Y}_N + \mat{A}_N) \right\|_2^2 \ll \frac{1}{N^2} \|\mat{Y}_N\|_2^2 + \frac{1}{N^2} \|\mat{A}_N\|_2^2. $$
By \eqref{eq:Aassump} and Lemma \ref{lemma:lln}, we have a.s.  
$$ \frac{1}{N^2} \|\mat{A}_N\|_2^2 = O(1) \quad \text{and} \quad \frac{1}{N^2} \|\mat{Y}_N \|_2^2 = O(1), $$
and \eqref{eq:hsbnd} follows.  

Let $B > 0$.  Define
$$ T:= \left\{(s,t) : |s| \leq B, |t| \leq B^3 \right\}. $$
By \cite[Lemma 11.7]{BSbook} and \eqref{eq:hsbnd}, in order to prove Theorem \ref{thm:main} it suffices to show that for each fixed $B>0$ a.s.
$$ \iint_T \left[g_{\frac{1}{\sqrt{N}}(\mat{Y}_N + \mat{A}_N)}(s,t) - g(s,t)\right] e^{\sqrt{-1} us + \sqrt{-1} vt} dt ds \longrightarrow 0 $$
as $N \rightarrow \infty$.  

Let $A > 0$ be the constant from Theorem \ref{thm:least-sing-value}, and set $\eps_N := N^{-A}$.  Recall that $z := s + \sqrt{-1} t$.  Following the integration by parts argument from \cite[Section 11.7]{BSbook}, it suffices to show that a.s.
\begin{equation} \label{eq:showintdiff}
	 \limsup_{N \rightarrow \infty} \iint_{T} \left| \int_{\eps_N}^{\infty} \log |x|^2 \left(\nu_{\frac{1}{\sqrt{N}} (\mat{Y}_N + \mat{A}_N) - z\mat{I}}(dx) - \nu_z(dx) \right) \right| dt ds = 0 
\end{equation}
and
\begin{equation} \label{eq:showintsingle}
	\limsup_{N \rightarrow \infty} \iint_{T} \left| \int_{0}^{\eps_N} \log |x|^2 \nu_{\frac{1}{\sqrt{N}} (\mat{Y}_N + \mat{A}_N) - z\mat{I}}(dx) \right| dt ds = 0,
\end{equation}
and similarly with the two-dimensional integral on $T$ replaced by one-dimensional integrals on the boundary of $T$.  We shall only estimate the two-dimensional integrals, as the treatment of the one-dimensional integrals are similar.  

We prove \eqref{eq:showintdiff} first.  By \eqref{eq:hsbnd}, it follows that $\nu_{\frac{1}{\sqrt{N}} (\mat{Y}_N + \mat{A}_N) - z \mat{I}}$ is supported on $[-N^{50}, N^{50}]$ a.s.  Thus, it suffices to show that a.s.
$$ \limsup_{N \rightarrow \infty} \iint_{T} \left| \int_{\eps_N}^{N^{50}} \log |x|^2 \left(\nu_{\frac{1}{\sqrt{N}} (\mat{Y}_N + \mat{A}_N) - z\mat{I}}(dx) - \nu_z(dx) \right) \right| dt ds = 0. $$
By definition of $\eps_N$, it suffices to show that a.s.
\begin{equation} \label{eq:nalphanorm}
	\limsup_{N \rightarrow \infty} (\log N) \sup_{|z| \leq M} \sup_{x \in \mathbb{R}} \left| \nu_{\frac{1}{\sqrt{N}} (\mat{Y}_N + \mat{A}_N ) - z\mat{I}}( (-\infty, x)) - \nu_z( (-\infty, x)) \right| = 0,
\end{equation}
where $M := B \sqrt{1+B^4}$.  \eqref{eq:nalphanorm} will follow from Lemma \ref{lemma:rate} below.

We now prove \eqref{eq:showintsingle}.  By Theorem \ref{thm:least-sing-value}, we have
\begin{equation} \label{eq:svzero}
	\text{for a.e. } z \in T, \text{ a.s. }\quad \lim_{N \rightarrow \infty} \int_{0}^{\eps_N} \log |x|^2 \nu_{\frac{1}{\sqrt{N}} (\mat{Y}_N + \mat{A}_N) - z\mat{I}}(dx) = 0.
\end{equation}

We note that it is possible to switch the quantifiers ``a.e.'' on $z$ and ``a.s.'' on $\omega$ in \eqref{eq:svzero} using Fubini's theorem and the arguments from \cite[Section 4]{BC}, where $\omega$ denotes an element of the sample space.  Thus, we have
\begin{equation} \label{eq:svzeroreverse}
	\text{a.s., } \text{for a.e. } z \in T, \quad \lim_{N \rightarrow \infty} \int_{0}^{\eps_N} \log |x|^2 \nu_{\frac{1}{\sqrt{N}} (\mat{Y}_N + \mat{A}_N) - z\mat{I}}(dx) = 0.
\end{equation}

Using the $L^2$-norm argument in \cite[Section 12]{TVcirc} (more specifically, see the argument following \cite[equation (49)]{TVcirc}), it follows that a.s.
\begin{equation} \label{eq:unifbnd}
	\left( \iint_{T} \left| \int_{0}^{\eps_N} \log |x|^2 \nu_{\frac{1}{\sqrt{N}} (\mat{Y}_N + \mat{A}_N) - z \mat{I}}(dx) \right|^2 dt ds \right)^{1/2} 
\end{equation}
is bounded uniformly in $N$, and hence the sequence of functions 
$$ \int_{0}^{\eps_N} \log |x|^2 \nu_{\frac{1}{\sqrt{N}} (\mat{Y}_N + \mat{A}_N) - z\mat{I}}(dx) $$ 
is a.s. uniformly integrable on $T$.  Let $L \gg 1$ be a large parameter and define $T_{L,N}$ to be the set of all $z \in T$ such that $\left| \int_{0} ^{\eps_N} \log |x|^2 \nu_{\frac{1}{\sqrt{N}} (\mat{Y}_N + \mat{A}_N) - z \mat{I}}(dx)\right| \leq L$.  By \eqref{eq:svzeroreverse} and the dominated convergence theorem, we have a.s.
$$ \lim_{N \rightarrow \infty} \iint_{T_{L,N}} \left| \int_{0}^{\eps_N} \log |x|^2 \nu_{\frac{1}{\sqrt{N}}(\mat{Y}_N + \mat{A}_N) - z\mat{I}}(dx) \right| dt ds = 0. $$
On the other hand, from the uniform boundedness of \eqref{eq:unifbnd}, we obtain a.s.
$$ \limsup_{N \rightarrow \infty} \iint_{T \setminus T_{L,N}} \left| \int_{0}^{\eps_N} \log |x|^2 \nu_{\frac{1}{\sqrt{N}}(\mat{Y}_N + \mat{A}_N) - z\mat{I}}(dx) \right| dt ds \ll \frac{1}{L}.  $$
Combining the bounds above and taking $L \rightarrow \infty$ yields \eqref{eq:showintsingle}.  

It remains to establish the following lemma.  

\begin{lemma} \label{lemma:rate}
Under the assumptions of Theorem \ref{thm:main}, there exists $\alpha > 0$ such that a.s.
$$ \sup_{|z| \leq M} \left\| \nu_{\frac{1}{\sqrt{N}}(\mat{Y}_N + \mat{A}_N) - z\mat{I}} - \nu_z \right\| = O_{M,M_{2+\tau}}(N^{-\alpha}), $$
where $\| \nu - \mu \| := \sup_{x \in \mathbb{R}} | \nu((-\infty, x)) - \mu((-\infty, x))|$ for any two probability measures $\nu, \mu$ on the real line and
$$ M_{2+\tau} := \sum_{k=1}^m \left( \E|\xi_{k,1}|^{2+\tau} + \E|\xi_{k,2}|^{2+\tau} \right). $$ 
\end{lemma}

In order to prove Lemma \ref{lemma:rate}, we will need the following results from \cite{BSbook}.  

\begin{theorem}[Theorem A.43 from \cite{BSbook}] \label{thm:BS:A43}
Let $\mat{A}$ and $\mat{B}$ be two $n \times n$ Hermitian matrices.  Then
$$ \| F^\mat{A} - F^\mat{B} \| \leq \frac{1}{n} \rank(\mat{A} - \mat{B}), $$
where $\|f\| = \sup_{x} |f(x)|$.  
\end{theorem}

\begin{lemma}[Lemma B.18 from \cite{BSbook}] \label{lemma:BS:B18}
If $G$ satisfies $\sup_x |G(x+y) - G(x)| \leq D |y|^\alpha$ for all $y$, then
$$ L(F,G) \leq \|F - G\| \leq (D+1) L^{\alpha}(F,G) $$
for all $F$.  Here $L(F,G)$ denotes the Levy distance, defined in \eqref{eq:def:levy}, between the distribution functions $F$ and $G$. 
\end{lemma}

\begin{proof}[Proof of Lemma \ref{lemma:rate}]
The proof of the lemma is based on the arguments from \cite[Lemma 64]{TVuniv}; similar arguments were also used in the proof of \cite[Lemma 7.14]{NgOCirc}. 

By Theorem \ref{thm:BS:A43} and \eqref{eq:Aassump},
$$ \sup_{z \in \mathbb{C}} \left\| \nu_{\frac{1}{\sqrt{N}} (\mat{Y}_N + \mat{A}_N) - z\mat{I}} - \nu_{\frac{1}{\sqrt{N}} \mat{Y}_N - z\mat{I}} \right\| = O(N^{-\eps}). $$
Thus, by the triangle inequality, it suffices to show that a.s.
$$ \sup_{|z| \leq M} \left\| \nu_{\frac{1}{\sqrt{N}} \mat{Y}_N - z\mat{I}} - \nu_z \right\| = O_{M,M_{2+\tau}}(N^{-\alpha}) $$
for some $\alpha > 0$.  

In view of Lemma \ref{lemma:nuz}, it follows that for each $z \in \mathbb{C}$, $\nu_z$ has density $\rho_z$ with 
$$ \sup_{z \in \mathbb{C}} \sup_{x \in \mathbb{R}} |\rho_z(x)| \leq 1. $$
Thus, by Lemma \ref{lemma:BS:B18}, it suffices to show that a.s.
$$ \sup_{|z| \leq M} L\left(F^{(N)}_z, F_z \right) = O_{M,M_{2+\tau}}(N^{-\alpha}), $$
where $F_z$ is the cumulative distribution function of $\nu_z$.  
We remind the reader that $L(F,G)$ denotes the Levy distance, defined in \eqref{eq:def:levy}, between the distribution functions $F$ and $G$.  

By Lemma \ref{lemma:truncn}, it suffices to show that a.s.
\begin{equation} \label{eq:showsupzmnorm}
	\sup_{|z| \leq M} \left\| \nu_{\frac{1}{\sqrt{N}} \hat{\mat{Y}}_N - z \mat{I}} - \nu_z \right\| = O_{M,M_{2+\tau}}(N^{-\alpha}),
\end{equation}
where $\hat{\mat{Y}}_N$ is the truncated matrix from Lemma \ref{lemma:truncn} for some $0 < \delta < \min\{1/100, 1/\tau\}$.  Recall that the matrix $\mat{q}$ is a function of $\eta$ and $z$.  From Lemma \ref{lemma:nuz}, we find that $a(\mat{q})$ is the Stieltjes transform of $\nu_z$.  That is,
$$ a(\mat{q}) = \int \frac{1}{x - \eta} d \nu_z(dx) = \int \frac{ \rho_z(x)}{x - \eta} dx. $$

By Lemma \ref{lemma:nuz}, we choose $\beta > 100$ sufficiently large (depending only on $M$) such that $\rho_z$ is supported inside the interval $[-\beta/2, \beta/2]$ for all $|z| \leq M$.  By Lemma \ref{lemma:concentrate} and Lemma \ref{approxthm}, it follows that a.s.
\begin{equation} \label{eq:diststtransf}
	\sup_{|z| \leq M } \sup_{|\eta| \leq 4\beta, \Im(\eta) \geq v_N} | \hat{a}_N(\mat{q}) - a(\mat{q})| = O_{M,M_{2+\tau}}(v_N^4), 
\end{equation}
where $v_N:= N^{-\delta \tau/1000}$, $\hat{a}_N(\mat{q}) := \frac{1}{2m} \tr \hat{\mat{\Gamma}}_N$, and $\hat{\mat{\Gamma}}_N$ is defined identically to $\mat{\Gamma}_N$ except with the matrix $\hat{\mat{H}}_N$ instead of $\mat{H}_N$.  

At this point, the proof of the lemma follows nearly verbatim the proof given in \cite[Lemma 7.14]{NgOCirc}.  The only changes required are notational; we omit the details.  
\end{proof}

\section{Proof of Theorem \ref{thm:prodWigner}} \label{sec:Wigner}

This section is devoted to the proof of Theorem \ref{thm:prodWigner}.  As noted above, the proof of Theorem \ref{thm:prodWigner} is very similar to the proof of Theorem \ref{thm:main}.  We define the linearized matrix $\mat{Z}_N := \frac{1}{\sqrt{N}} \mat{Y}_N$, where
$$ \mat{Y}_N := \begin{bmatrix} 0 & \mat{Y}_{N,1} \\ \mat{Y}_{N,2} & 0 \end{bmatrix}. $$

As before, the proof of Theorem \ref{thm:prodWigner} reduces to showing that the ESD of $\mat{Z}_N$ converges to the circular law as $N$ tends to infinity.  

\begin{theorem} \label{thm:circularWigner}
Under the assumptions of Theorem \ref{thm:prodWigner}, the ESD of $\mat{Z}_N := \frac{1}{\sqrt{N}} \mat{Y}_N$ converges almost surely to the circular law $F_{\mathrm{circ}}$ as $N \to \infty$.
\end{theorem}

The proof that Theorem \ref{thm:prodWigner} follows from Theorem \ref{thm:circularWigner} is identical to the proof given in Section \ref{sec:linear}.   

The proof of Theorem \ref{thm:circularWigner} follows the same arguments outlined in Section \ref{sec:complete} (taking $\mat{A}_N = 0$).  Indeed, the proof in Section \ref{sec:complete} requires three key inputs: Lemma \ref{lemma:concentrate}, Theorem \ref{approxthm}, and Theorem \ref{thm:least-sing-value}.  Thus, in order to complete the proof, we will need versions of these results for the matrix $\mat{Z}_N$ defined above.  

We first observe that both Lemma \ref{lemma:concentrate} and Theorem \ref{approxthm} hold for the matrix $\mat{Z}_N$ defined above.  Indeed, the matrices $\mat{Y}_{N,1}$ and $\mat{Y}_{N,2}$ (as well as their truncated counterparts) are elliptic random matrices.  In fact, the proofs of Lemma \ref{lemma:concentrate} and Theorem \ref{approxthm} do not require any conditions on the correlations $\rho_k$.  Therefore, it only remains to prove the following analogue of Theorem \ref{thm:least-sing-value}.  

\begin{theorem}[Least singular value bound] \label{thm:least-sing-valueWigner}
Under the assumptions of Theorem \ref{thm:prodWigner}, there exists $A >0$ such that for almost every $z \in \mathbb{C}$, almost surely
$$ \lim_{N \to \infty} \indicator{\sigma_{2N}( \mat{Z}_N - z\mat{I}) \leq N^{-A}} = 0. $$
\end{theorem}

In order to prove Theorem \ref{thm:least-sing-valueWigner}, we will need the following result due to Nguyen \cite{NgLSV}.  The version stated in \cite{NgLSV} requires that the deterministic matrices $\mat{F}_{N,k}$ be real symmetric.  However, the proof given in \cite{NgLSV} only requires that $\mat{F}_{N,k}$ be complex symmetric \cite{NgPC}.  In particular, the proof only uses that the $(i,j)$-entry of $\mat{F}_{N,k}$ be equal to the $(j,i)$-entry.  We present the most general version below.  

\begin{theorem}[Theorem 1.5 from \cite{NgLSV}] \label{thm:NgLSV}
For each $k=1,2$, let $\mat{F}_{N,k}$ be a $N \times N$ complex symmetric matrix whose entries are bounded in magnitude by $N^\alpha$, for some $\alpha > 0$.  Then, under the assumptions of Theorem \ref{thm:prodWigner}, for any $B > 0$, there exists $A>0$ (depending on $\alpha, B$) such that
$$ \Prob \left( \min_{k=1,2} \sigma_{N}( \mat{Y}_{N,k} + \mat{F}_{N,k}) \leq N^{-A} \right) = O(N^{-B}). $$
\end{theorem}

We now prove Theorem \ref{thm:least-sing-valueWigner}

\begin{proof}[Proof of Theorem \ref{thm:least-sing-valueWigner}]
Let $A$ be a large positive constant to be chosen later.  By the Borel-Cantelli lemma, it suffices to show that, for almost every $z \in \mathbb{C}$,
$$ \Prob \left( \left\| (\mat{Z}_N - z \mat{I})^{-1} \right\| \geq N^{A} \right) = O(N^{-2}). $$
By \eqref{eq:schur}, we observe that, for $z \neq 0$, $(\mat{Z}_N - z \mat{I})^{-1}$ has the form
$$ \begin{bmatrix} z \left(\frac{1}{N} \mat{Y}_{N,1} \mat{Y}_{N,2} - z^2 \mat{I} \right)^{-1} & \left(\frac{1}{N} \mat{Y}_{N,1} \mat{Y}_{N,2} - z^2 \mat{I} \right)^{-1} \frac{1}{\sqrt{N}} \mat{Y}_{N,1} \\ z^2 \left(\frac{1}{N} \mat{Y}_{N,2} \mat{Y}_{N,1} - z^2 \mat{I} \right)^{-1} \frac{1}{\sqrt{N}} \mat{Y}_{N,2} & z \left( \frac{1}{N} \mat{Y}_{N,2} \mat{Y}_{N,1} - z^2 \mat{I} \right)^{-1} \end{bmatrix} $$
provided the relevant inverses exist.  Thus, it suffices to show that the spectral norm of each block above is $O(N^{A})$ with probability $1-O(N^{-2})$.  The treatment of each block is similar; as an illustration, we will show that
$$ \Prob \left( \left\| \left(\frac{1}{N} \mat{Y}_{N,1} \mat{Y}_{N,2} - z^2 \mat{I} \right)^{-1} \frac{1}{\sqrt{N}} \mat{Y}_{N,1} \right\| \geq N^{A} \right) = O(N^{-2}). $$

Indeed, since 
$$ \Prob \left( \left\| \frac{1}{\sqrt{N}} \mat{Y}_{N,1} \right\| \geq N^{100} \right) \leq N^{-200} \E \left\| \frac{1}{\sqrt{N}} \mat{Y}_{N,1} \right\|^2_2 = O(N^{-2}), $$
it suffices to show that there exists $A>200$ such that
\begin{equation} \label{eq:show2prodbnd}
	\Prob \left( \left\| \left( \frac{1}{N} \mat{Y}_{N,1} \mat{Y}_{N,2} - z^2 \mat{I} \right)^{-1} \right\| \geq N^{A} \right) = O(N^{-2}). 
\end{equation}

By Theorem \ref{thm:NgLSV}, there exists $A'>0$ such that the event
$$ \Omega_N := \left\{ \left\| \left( \frac{1}{\sqrt{N}} \mat{Y}_{N,2} \right)^{-1} \right\| \leq N^{A'} \right\} $$
holds with probability $1-O(N^{-2})$.  On this event, we observe that
$$ \left( \frac{1}{N} \mat{Y}_{N,1} \mat{Y}_{N,2} - z^2 \mat{I} \right)^{-1} = \left( \frac{1}{\sqrt{N}} \mat{Y}_{N,2} \right)^{-1} \left( \frac{1}{\sqrt{N}} \mat{Y}_{N,1} - z^2 \left(\frac{1}{\sqrt{N}} \mat{Y}_{N,2} \right)^{-1} \right)^{-1}. $$
Therefore, we conclude that
\begin{align*}
	\Prob &\left( \left\| \left( \frac{1}{N} \mat{Y}_{N,1} \mat{Y}_{N,2} - z^2 \mat{I} \right)^{-1} \right\| \geq N^{A} \right) \\
	&\leq \Prob \left( \left\| \left( \frac{1}{N} \mat{Y}_{N,1} \mat{Y}_{N,2} - z^2 \mat{I} \right)^{-1} \right\| \geq N^{A} \bigg| \Omega_N \right) \Prob(\Omega_N) + \Prob(\Omega_N^C) \\
	&\leq O(N^{-2}) + \Prob \left( \left\| \left( \frac{1}{\sqrt{N}} \mat{Y}_{N,1} - z^2 \left( \frac{1}{\sqrt{N}} \mat{Y}_{N,2}\right)^{-1} \right)^{-1} \right\| \geq N^{A/2} \bigg| \Omega_N \right)
\end{align*} 
for $A$ sufficiently large.  

We now recall that $\mat{Y}_{N,1}$ and $\mat{Y}_{N,2}$ are independent, and, on the event $\Omega_N$, the entries of $\left( \frac{1}{\sqrt{N}} \mat{Y}_{N,2} \right)^{-1}$ are bounded in magnitude by $N^{A'}$.  In addition, we observe that $z^2 \left( \frac{1}{\sqrt{N}} \mat{Y}_{N,2} \right)^{-1}$ is a complex symmetric matrix since $\mat{Y}_{N,2}$ is a real symmetric matrix.  Therefore, by Theorem \ref{thm:NgLSV}, there exists $A>0$ such that
$$ \Prob \left( \left\| \left( \frac{1}{\sqrt{N}} \mat{Y}_{N,1} - z^2 \left( \frac{1}{\sqrt{N}} \mat{Y}_{N,2}\right)^{-1} \right)^{-1} \right\| \geq N^{A/2} \bigg| \Omega_N \right) = O(N^{-2}). $$
This verifies \eqref{eq:show2prodbnd}, and hence the proof of Theorem \ref{thm:least-sing-valueWigner} is complete.  
\end{proof}

\appendix

\section{Truncation} \label{sec:truncationproof}

This section contains somewhat standard proofs of the truncation results in Section \ref{sec:truncation}.  

\begin{proof}[Proof of Lemma \ref{lemma:truncation}]
We begin by observing that 
\begin{equation} \label{eq:1var}
	\var(\tilde{\xi}_i^{(N)}) = \E|\tilde{\xi}_i^{(N)}|^2 \leq \E|\xi_i \indicator{|\xi_i| \leq N^\delta}|^2 \leq 1 
\end{equation}
for $i=1,2$.  We also have
\begin{align*}
	\left| 1 - \var(\tilde{\xi}_i^{(N)}) \right| \leq 2 \E |\xi_i|^2 \indicator{|\xi_i| > N^\delta} \leq 2 \frac{M_{2+\tau}}{N^{\delta\tau}},
\end{align*}
which verifies property \eqref{item:varbnd}.  

Thus, we take $N_0$ sufficiently large such that
\begin{equation} \label{eq:varbnd}
	\var(\tilde{\xi}_i^{(N)}) \geq 1/2 
\end{equation}
for all $i=1,2$ and $N \geq N_0$.  Then \eqref{item:meanvar} follows by construction and the bound in \eqref{eq:varbnd}.  

It remains to prove \eqref{item:rho}.  Set $\tilde{\rho}^{(N)} := \E[ \tilde{\xi}_1^{(N)} \tilde{\xi}_2^{(N)}]$.  Then by the Cauchy-Schwarz inequality and \eqref{eq:1var}, we have
\begin{align*}
	|\tilde{\rho}^{(N)} - \rho| &\leq \sqrt{ \E |\xi_1|^2 \indicator{|\xi_1| > N^\delta}} \sqrt{ \E|\xi_2|^2 \indicator{|\xi_2| > N^\delta}} + \sum_{i=1}^2 \left( \sqrt{ \E |\xi_i|^2 \indicator{|\xi_i| > N^\delta}} +  \E |\xi_i|^2 \indicator{|\xi_i| > N^\delta} \right) \\
		&\leq 5 \frac{M_{2+\tau}}{N^{\delta \tau/2}}. 
\end{align*}
By property \eqref{item:varbnd} and \eqref{eq:1var}, we obtain
\begin{align*}
	|\hat{\rho}^{(N)} - \tilde{\rho}^{(N)}| &\leq \E |\tilde{\xi}_1^{(N)} \tilde{\xi}_2^{(N)}| \left| \frac{1}{\sqrt{ \var(\tilde{\xi}_1^{(N)}) \var(\tilde{\xi}_2^{(N)})} } - 1 \right| \\
		&\leq 2 \left| \sqrt{ \var(\tilde{\xi}_1^{(N)}) \var(\tilde{\xi}_2^{(N)})} - 1 \right| \\
		&\leq 2 \left|\var(\tilde{\xi}_1^{(N)}) \var(\tilde{\xi}_2^{(N)}) - 1 \right| \\
		&\leq 2 \left| \var(\tilde{\xi}_1^{(N)}) - 1 \right| + 2 \left| \var(\tilde{\xi}_2^{(N)}) - 1 \right| \\
		&\leq 8 \frac{M_{2+\tau}}{N^{\delta \tau}}. 
\end{align*}
Combining the bounds above completes the proof of property \eqref{item:rho}.  
\end{proof}

\begin{proof}[Proof of Lemma \ref{lemma:lln}]
We begin with \eqref{eq:limsupYN}.  By the block structure of $\mat{Y}_N$, it suffices to show that a.s.
\begin{equation} \label{eq:limsupYNshow}
	\limsup_{N \to \infty} \frac{1}{N^2} \| \mat{Y}_{N,k} \|_2^2 < \infty 
\end{equation}
for $k = 1,\ldots,m$.  We now decompose
\begin{align*}
	\| \mat{Y}_{N,k} \|_2^2 &= \sum_{1 \leq i < j \leq N} | (\mat{Y}_{N,k})_{ij} |^2 + \sum_{1 \leq j < i \leq N} | (\mat{Y}_{N,k})_{ij} |^2 + \sum_{i=1}^N | (\mat{Y}_{N,k})_{ii} |^2,
\end{align*}
where the summands in each sum are iid copies of $\xi_{k,1}, \xi_{k,2}$, and $\zeta_k$, respectively.  Thus, applying the law of large numbers to each sum yields \eqref{eq:limsupYNshow}.  

For \eqref{eq:limsuphat}, we decompose 
\begin{align*}
	\| \hat{\mat{Y}}_N \|_2^2 &= \sum_{k=1}^m \| \hat{\mat{Y}}_{N,k} \|^2_2 \\
		&\leq 2 \sum_{k=1}^m \sum_{i,j=1}^N | (\tilde{\mat{Y}}_{N,k})_{ij} |^2 \\
		& \leq 4 \sum_{k=1}^m  \sum_{i,j=1}^N \left( |(\mat{Y}_{N,k})_{ij} |^2 + \E |(\mat{Y}_{N,k})_{ij} |^2\right) \\
		& \leq 4 \left( \| \mat{Y}_N \|_2^2 + \E \| \mat{Y}_N \|_2^2 \right)
\end{align*}
by Lemma \ref{lemma:truncation}.  As the atom variables have finite variance, \eqref{eq:limsuphat} follows from \eqref{eq:limsupYN}.  

It remains to prove \eqref{eq:limdiff}.  By the triangle inequality and the block structure of $\mat{Y}_N$ and $\hat{\mat{Y}}_N$, it suffices to show that almost surely, for $k=1,\ldots,m$,
\begin{equation} \label{eq:twolims}
	\lim_{N \to \infty} \frac{N^{\delta \tau}}{N^2} \| \mat{Y}_{N,k} - \tilde{\mat{Y}}_{N,k} \|_2^2 = 0 \quad \text{and} \quad \lim_{N \to \infty} \frac{N^{\delta \tau}}{N^2} \| \tilde{\mat{Y}}_N - \hat{\mat{Y}}_N \|_2^2 = 0.
\end{equation}

Fix $1 \leq k \leq m$.  Then we have
\begin{align*}
	\| \mat{Y}_{N,k} - \tilde{\mat{Y}}_{N,k} \|_2^2 \leq 2 \sum_{i,j=1}^N \left(  |(\mat{Y}_{N,k})_{ij}|^2 \indicator{|(\mat{Y}_{N,k})_{ij}| > N^\delta} + \E  |(\mat{Y}_{N,k})_{ij}|^2  \indicator{|(\mat{Y}_{N,k})_{ij}| > N^\delta}\right).  
\end{align*}
Since $\delta \tau < 1$, we have that a.s.
\begin{align*}
	\limsup_{N \to \infty} \frac{N^{\delta \tau}}{N^2} &\sum_{i=1}^N \left(  |(\mat{Y}_{N,k})_{ii}|^2 \indicator{|(\mat{Y}_{N,k})_{ij}| > N^\delta} + \E  |(\mat{Y}_{N,k})_{ii}|^2 \right) \\
		&\leq \limsup_{N \to \infty} \frac{N^{\delta \tau} }{N^2} \sum_{i=1}^N \left( |(\mat{Y}_{N,k})_{ii}|^2 + \E |(\mat{Y}_{N,k})_{ii}|^2 \right) = 0
\end{align*}
by the law of large numbers.  On the other hand, 
\begin{align*} 
	\frac{N^{\delta \tau}}{N^2} & \sum_{1 \leq i < j \leq N} \left(  |(\mat{Y}_{N,k})_{ij}|^2 \indicator{|(\mat{Y}_{N,k})_{ij}| > N^\delta} + \E  |(\mat{Y}_{N,k})_{ij}|^2  \indicator{|(\mat{Y}_{N,k})_{ij}| > N^\delta} \right) \\
		&\leq \frac{1}{N^2} \sum_{1 \leq i < j \leq N} \left(  |(\mat{Y}_{N,k})_{ij}|^{2+\tau} \indicator{|(\mat{Y}_{N,k})_{ij}| > N^\delta} + \E  |(\mat{Y}_{N,k})_{ij}|^{2+\tau}  \indicator{|(\mat{Y}_{N,k})_{ij}| > N^\delta} \right).
\end{align*}
By the dominated convergence theorem 
\begin{align*}
	\limsup_{N \to \infty} \frac{1}{N^2} &\E  |(\mat{Y}_{N,k})_{ij}|^{2+\tau}  \indicator{|(\mat{Y}_{N,k})_{ij}| > N^\delta}
		\leq \limsup_{N \to \infty}  \E |\xi_{k,1}|^{2+\tau} \indicator{|\xi_{k,1}| > N^{\delta}} = 0. 
\end{align*}
Furthermore, by the law of large numbers, we have a.s.
$$ \limsup_{N \to \infty} \frac{1}{N^2} \sum_{1 \leq i < j \leq N}   |(\mat{Y}_{N,k})_{ij}|^{2+\tau} \indicator{|(\mat{Y}_{N,k})_{ij}| > N^\delta} = 0. $$
The sum involving the indices $1 \leq j < i \leq N$ is handled similarly, and hence we conclude that a.s. 
$$  \lim_{N \to \infty} \frac{N^{\delta \tau}}{N^2} \| \mat{Y}_{N,k} - \tilde{\mat{Y}}_{N,k} \|_2^2 = 0. $$

We now consider the second limit in \eqref{eq:twolims}.  By Lemma \ref{lemma:truncation}, we obtain 
\begin{align*}
	\| \tilde{\mat{Y}}_{N} - \hat{\mat{Y}}_N \|_2^2 &= \sum_{k=1}^m \sum_{i,j=1}^N |(\hat{\mat{Y}}_{N,k})_{ij}|^2 \left| \sqrt{ \var( (\tilde{\mat{Y}}_{N,k})_{ij}) } - 1 \right|^2 \\
		&\leq \sum_{k=1}^m \sum_{i,j=1}^N |(\hat{\mat{Y}}_{N,k})_{ij}|^2 | \var( (\tilde{\mat{Y}}_{N,k})_{ij}) - 1 |^2 \\
		&\leq \frac{1}{N^{2 \delta \tau}} \sum_{k=1}^m \sum_{i,j=1}^N |(\hat{\mat{Y}}_{N,k})_{ij}|^2 \\
		&\leq \frac{1}{N^{2 \delta \tau}} \| \hat{\mat{Y}}_{N} \|_2^2.
\end{align*}
Thus, by \eqref{eq:limsuphat}, a.s. we have
$$ \limsup_{N \to \infty}  \frac{N^{\delta \tau}}{N^2} \| \tilde{\mat{Y}}_{N} - \hat{\mat{Y}}_N \|_2^2 \leq \limsup_{N \to \infty} \frac{1}{N^{2 + \delta \tau}} \| \hat{\mat{Y}}_N \|_2^2 = 0, $$
and the proof of the lemma is complete
\end{proof}

\end{document}